\documentclass[11pt,letterpaper,twoside]{article}
\usepackage{amsmath,amssymb,amsthm,amsfonts,amstext,amsbsy,amscd}

\usepackage{comment}
\usepackage{graphicx}
\usepackage{mathrsfs}
\usepackage{harvard}

\numberwithin{equation}{section}

\swapnumbers

\newtheorem{satz}{Satz}[section]

\newtheorem{theorem}[satz]{Theorem}
\newtheorem{proposition}[satz]{Proposition}
\newtheorem{corollary}[satz]{Corollary}

\newtheorem{definition}[satz]{Definition}

\newtheorem{remark}[satz]{Remark}

\DeclareMathOperator{\E}{{\mathbb E}}

\DeclareMathOperator{\R}{{\mathbb R}}

\DeclareMathOperator{\Z}{{\mathbb Z}}
\DeclareMathOperator{\N}{{\mathbb N}}

\DeclareMathOperator{\PP}{{\mathbb P}}
\DeclareMathOperator{\QQ}{{\mathbb Q}}

\DeclareMathOperator{\supp}{supp}

 \DeclareMathOperator{\Id}{Id}

\DeclareMathOperator{\Var}{Var}

\providecommand{\eps}{\varepsilon}
\renewcommand{\phi}{\varphi}
\renewcommand{\theta}{\vartheta}

\renewcommand{\cdot}{{\scriptstyle \bullet} }
\providecommand{\abs}[1]{\lvert #1 \rvert}
\providecommand{\norm}[1]{\lVert #1 \rVert}

\providecommand{\babs}[1]{{\Bigl\lvert #1 \Bigr\rvert}}
\providecommand{\scapro}[2]{\langle #1,#2 \rangle}
\providecommand{\floor}[1]{\lfloor #1 \rfloor}

\renewcommand{\le}{\leqslant}
\renewcommand{\ge}{\geqslant}

\newcommand{\cit}[1]{\citeasnoun{#1}}

\renewcommand{\cal}{\mathscr}     

\begin{document}
\title{Asymptotic equivalence and sufficiency for volatility estimation under microstructure noise}
\author{\parbox[t]{6.5cm}{\centering Markus
 Rei\ss\\[2mm]
 \normalsize{\it
Institute of Mathematics\\
Humboldt-Universit\"at zu Berlin}\\
\mbox{} mreiss@mathematik.hu-berlin.de}} \maketitle

\begin{abstract}
The basic model for high-frequency data in finance is considered, where an efficient price process is observed under microstructure noise. It is shown that this nonparametric model is in Le Cam's sense asymptotically equivalent to a Gaussian shift experiment in terms of the square root of the volatility function $\sigma$. As an application, simple rate-optimal estimators of the volatility and efficient estimators of the integrated volatility are constructed.
\end{abstract}

{\small \noindent {\it Key words and Phrases:} High-frequency data, integrated volatility, spot volatility estimation, Le Cam deficiency,
equivalence of experiments, Gaussian shift.

\newpage

\section{Introduction}

In recent years volatility estimation from high-frequency data has attracted a lot of attention in financial econometrics and statistics. Due to empirical evidence that the observed transaction prices of assets cannot follow a semi-martingale model, a prominent approach is to model the observations as the superposition of the true (or efficient) price process with some measurement error, conceived as microstructure noise. The main features are already present in the basic model of observing
\begin{equation}\label{EqE1Obs}
Y_i=X_{i/n}+\eps_i,\quad i=1,\ldots,n,
\end{equation}
with an efficient price process $X_t=\int_0^t\sigma(s)\,dB_s$, $B$ a standard Brownian motion, and $\eps_i\sim N(0,\delta^2)$ all independent. The aim is to perform
statistical inference on the volatility function $\sigma:[0,1]\to\R^+$, e.g. estimating the so-called integrated volatility $\int_0^1\sigma^2(t)\,dt$ over the trading day.

The mathematical foundation on the parametric formulation of this model has been laid by \cit{GloterJacodI} who prove the interesting result that the model is locally asymptotically normal (LAN) as $n\to\infty$, but with the unusual rate $n^{-1/4}$, while without microstructure noise the rate is $n^{-1/2}$. Starting with \cit{Zhangetal}, the nonparametric model has come into the focus of research. Mainly three different, but closely related approaches have been proposed afterwards to estimate the integrated volatility: multi-scale estimators \cite{Zhang}, realized kernels or autocovariances \cite{BNetal} and preaveraging \cite{Jacodetal}.
Under various degrees of generality, especially also for stochastic volatility, all authors provide central limit theorems with convergence rate $n^{-1/4}$ and an asymptotic variance involving the so-called quarticity $\int_0^1\sigma^4(t)\,dt$. Recently, also the problem of estimating the spot volatility $\sigma^2(t)$ itself has found some interest \cite{MunkJohannes}.

The aim of the present paper is to provide a thorough mathematical understanding of the basic model, to explain why statistical inference is not so canonical and to propose a simple estimator of the integrated volatility which is efficient. To this end we employ Le Cam's concept of asymptotic equivalence between experiments. In fact, our main theoretical result in Theorem \ref{ThmE0G0} states under some regularity conditions that observing $(Y_i)$ in \eqref{EqE1Obs} is for $n\to\infty$ asymptotically equivalent to observing the Gaussian shift experiment
\[ dY_t=\sqrt{2\sigma(t)}\,dt+\delta^{1/2}n^{-1/4}\,dW_t,\quad t\in [0,1],
\]
with Gaussian white noise $dW$. Not only the large noise level $\delta^{1/2}n^{-1/4}$ is apparent, but also a non-linear $\sqrt{\sigma(t)}$-form of the signal, from which optimal asymptotic variance results can be derived. Note that a similar form of a Gaussian shift was found to be asymptotically equivalent to nonparametric density estimation \cite{Nussbaum}. A key ingredient of our asymptotic equivalence proof are the results by \cit{GramaNussbaum02} on asymptotic equivalence for generalized nonparametric regression, but also ideas from \cit{Carter} and \cit{Reiss08} play a role. Moreover, fine bounds on Hellinger distances for Gaussian measures with different covariance operators turn out to be essential.

Roughly speaking, asymptotic equivalence means that any statistical inference procedure can be transferred from one experiment to the other such that the asymptotic risk remains the same, at least for bounded loss functions. Technically, two sequences of experiments ${\cal E}^n$ and ${\cal G}^n$, defined on possibly different sample spaces, but with the same parameter set, are asymptotically equivalent if the Le Cam distance $\Delta({\cal E}^n,{\cal G}^n)$ tends to zero. For ${\cal E}_i=({\cal X}_i,{\cal F}_i,(\PP_\theta^i)_{\theta\in\Theta})$, $i=1,2$, by definition, $\Delta({\cal E}_1,{\cal E}_2)=\max(\delta({\cal E}_1,{\cal E}_2),\delta({\cal E}_1,{\cal E}_2))$ holds in terms of the deficiency $\delta({\cal E}_1,{\cal E}_2)=\inf_M\sup_{\theta\in\Theta}\norm{MP_\theta^1-P_\theta^2}_{TV}$, where the infimum is taken over all randomisations or Markov kernels $M$ from $({\cal X}_1,{\cal F}_1)$ to $({\cal X}_2,{\cal F}_2)$, see e.g. \cit{LeCamYang} for details. In particular, $\delta({\cal E}_1,{\cal E}_2)=0$ means that ${\cal E}_1$ is more informative than ${\cal E}_2$ in the sense that any observation in ${\cal E}_2$ can be obtained from ${\cal E}_1$, possibly using additional randomisations. Here, we shall always explicitly construct the transformations and randomisations and we shall then only use that $\Delta({\cal E}_1,{\cal E}_2)\le \sup_{\theta\in\Theta}\norm{P_\theta^1-P_\theta^2}_{TV}$ holds when both experiments are defined on the same sample space.

The asymptotic equivalence is deduced stepwise. In Section \ref{SecRegWNM} the  regression-type model \eqref{EqE1Obs} is shown to be asymptotically equivalent to a corresponding white noise model with signal $X$. Then in Section \ref{SecLessInfG}, a very simple construction yields a Gaussian shift model with signal $\log(\sigma^2(\cdot)+c)$, $c>0$ some constant, which is asymptotically less informative, but only by a constant factor in the Fisher information. Inspired by this construction, we present a generalisation in Section \ref{SecSeqSimple} where the information loss can be made arbitrarily small (but not zero), before applying nonparametric local asymptotic theory in Section \ref{SecLoc} to derive asymptotic equivalence with our final Gaussian shift model for shrinking local neighbourhoods of the parameters. Section \ref{SecGlob} yields the global result, which is based on an asymptotic sufficiency result for simple independent statistics. 

Extensions and restrictions are discussed in Section \ref{SecDisc} before we use the theoretical insight to construct in Section \ref{SecAppl} a rate-optimal estimator of the spot volatility and an efficient estimator of the integrated volatility by a locally-constant approximation. Remarkably,  the asymptotic variance is found to depend on the third moment $\int_0^1\sigma^3(t)\,dt$ and for non-constant $\sigma^2(\cdot)$ our estimator outperforms previous approaches applied to the basic model. Constructions needed for the proof are presented and discussed alongside the mathematical results, deferring more technical parts to the Appendix, which in Section \ref{SecPrelim} also contains a summary of results on white noise models, the Hellinger distance and Hilbert-Schmidt norm estimates.

\section{The regression and white noise model}\label{SecRegWNM}

In the main part we shall work in the white noise setting, which is more intuitive to handle than the regression setting, which in turn is the observation model in practice. Let us define both models formally. For that we introduce the H\"older ball
\[ C_\alpha(R):=\{ f\in C^\alpha([0,1])\,|\,\norm{f}_{C^\alpha}\le R\}
\text{ with }
\norm{f}_{C^\alpha}=\norm{f}_\infty+\sup_{x\not=y}\frac{\abs{f(x)-f(y)}}{\abs{x-y}^\alpha}.
\]

\begin{definition}
Let ${\cal E}_0={\cal E}_0(n,\delta,\alpha,R,\underline{\sigma}^2)$ with $n\in\N$, $\delta>0$, $\alpha\in (0,1)$, $R>0$, $\underline{\sigma}^2\ge 0$ be the statistical experiment generated by observing \eqref{EqE1Obs}. The volatility $\sigma^2$ belongs to the class
\[ {\cal S}(\alpha,R,\underline{\sigma}^2):=\Big\{\sigma^2\in C^\alpha(R)\,\Big|\, \min_{t\in[0,1]}\sigma^2(t)\ge \underline{\sigma}^2\Big\}.
\]

Let ${\cal E}_1={\cal E}_1(\eps,\alpha,R,\underline{\sigma}^2)$ with $\eps>0$, $\alpha\in (0,1)$, $R>0$, $\underline{\sigma}^2\ge 0$ be the statistical experiment generated by observing
\[ dY_t=X_t\,dt+\eps\,dW_t,\quad t\in [0,1],\]
with $X_t=\int_0^t\sigma(s)\,dB_s$ as above, independent standard Brownian motions $W$ and $B$ and $\sigma^2\in {\cal S}(\alpha,R,\underline{\sigma}^2)$.
\end{definition}

From \cit{BrownLow} it is well known that the white noise and the Gaussian regression model are asymptotically equivalent for noise level $\eps=\delta/\sqrt{n}\to 0$ as $n\to\infty$, provided the signal is $\beta$-H\"older continuous for $\beta>1/2$.
Since Brownian motion and thus also our price process $X$ is only H\"older continuous of order $\beta<1/2$ (whatever $\alpha$ is), it is not clear whether asymptotic equivalence can hold for the experiments ${\cal E}_0$ and ${\cal E}_1$. Yet, this is true. Subsequently, we employ the notation $A_n\lesssim B_n$ if $A_n=O(B_n)$ and $A_n\thicksim B_n$ if $A_n\lesssim B_n$ as well as $B_n\lesssim A_n$ and obtain:

\begin{theorem}\label{ThmRegression}
For any $\alpha>0$, $\underline\sigma^2\ge 0$ and $\delta,R>0$ the experiments ${\cal E}_0$ and ${\cal E}_1$ with  $\eps=\delta/\sqrt{n}$ are asymptotically equivalent ; more precisely:
\[ \Delta({\cal E}_0(n,\delta,\alpha,R,\underline{\sigma}^2),
{\cal E}_1(\delta/\sqrt{n},h,\alpha,R,\underline{\sigma}^2))\lesssim R\delta^{-2}n^{-\alpha}.
\]
\end{theorem}

Interestingly, the asymptotic equivalence holds for any positive H\"older regularity $\alpha>0$. In particular, the volatility $\sigma^2$ could be itself a continuous semi-martingale, but such that $X$ conditionally on $\sigma^2$ remains Gaussian. As the proof in Section \ref{SecProofThmRegression} of the appendix reveals, we construct the equivalence by rate-optimal approximations of the anti-derivative of $\sigma^2$ which lies in $C^{1+\alpha}$. Similar techniques have been used by \cit{Carter} and \cit{Reiss08}, but here we have to cope with the random signal for which we need to bound the Hilbert-Schmidt norm of the respective covariance operators. Note further that the asymptotic equivalence even holds when the level of the microstructure noise $\delta$ tends to zero, provided $\delta^2n^\alpha\to \infty$ remains valid.

\section{Less informative Gaussian shift experiments}\label{SecLessInfG}

From now on we shall work with the white noise observation experiment ${\cal E}_1$, where the main structures are more clearly visible. In this section we shall find easy Gaussian shift models which are asymptotically not more informative than ${\cal E}_1$, but already permit rate-optimal estimation results.  The whole idea is easy to grasp once we can replace the volatility $\sigma^2$ by a piecewise constant approximation on small blocks of size $h$. That this is no loss of generality, is shown by the subsequent asymptotic equivalence result, proved in Section \ref{SecPropPiecewiseConstant} of the appendix.

\begin{definition}
Let ${\cal E}_2={\cal E}_2(\eps,h,\alpha,R,\underline{\sigma}^2)$ be the statistical experiment generated by observing
\[ dY_t=X_t^h\,dt+\eps\,dW_t,\quad t\in [0,1],\]
with $X_t^h=\int_0^t\sigma(\floor{s}_h)\,dB_s$, $\floor{s}_h:=\floor{s/h}h$ for $h>0$ and $h^{-1}\in\N$,  and independent standard Brownian motions $W$ and $B$. The volatility $\sigma^2$ belongs to the class ${\cal S}(\alpha,R,\underline{\sigma}^2)$.
\end{definition}

\begin{proposition}\label{PropPiecewiseConstant}
Assume $\alpha>1/2$ and $\underline\sigma^2>0$. Then for $\eps\to 0$, $h^\alpha=o(\eps^{1/2})$ the experiments ${\cal E}_1$ and ${\cal E}_2$ are asymptotically equivalent ; more precisely:
\[ \Delta({\cal E}_1(\eps,\alpha,R,\underline{\sigma}^2),
{\cal E}_2(\eps,h,\alpha,R,\underline{\sigma}^2))\lesssim R \underline{\sigma}^{-3/2} h^\alpha \eps^{-1/2}.
\]
\end{proposition}

In the sequel we always assume $h^\alpha=o(\eps^{1/2})$ to hold such that we can work equivalently with ${\cal E}_2$. Recall that observing $Y$ in a white noise model is equivalent to observing $(\int e_m\,dY)_{m\ge 1}$ for an orthonormal basis $(e_m)_{m\ge 1}$ of $L^2([0,1])$, cf. also Subsection \ref{SecPrelim} below. Our first step is thus to find an orthonormal system  (not a basis) which extracts as much {\it local } information on $\sigma^2$ as possible. For any $\phi\in L^2([0,1])$ with $\norm{\phi}_{L^2}=1$ we have by partial integration
\begin{align}
\int_0^1 \phi(t) dY_t&=\int_0^1 \phi(t)X_t^h\,dt+\eps\int_0^1\phi(t)\,dW_t\nonumber\\
&=\Phi(1)X_1^h-\Phi(0)X_0^h-\int_0^1 \Phi(t)\sigma(\floor{t}_h)\,dB_t+\eps\int\phi(t)\,dW_t\nonumber\\
&=\Big(\int_0^1\Phi^2(t)\sigma^2(\floor{t}_h)\,dt+\eps^2\Big)^{1/2}\zeta_\phi\label{EqCovY}
\end{align}
where $\Phi(t)=-\int_t^1\phi(s)\,ds$ is the antiderivative of $\phi$ with $\Phi(1)=0$ and $\zeta_\phi\sim N(0,1)$ holds. To ensure that $\Phi$ has only support in some interval $[kh,(k+1)h]$, we require $\phi$ to have support in $[kh,(k+1)h]$ and to satisfy $\int \phi(t)\,dt=0$. The function $\phi_k$ with $\supp(\phi_k)=[kh,(k+1)h]$, $\norm{\phi_k}_{L^2}=1$, $\int \phi_k(t)\,dt=0$ that maximizes the information load $\int \Phi_k^2(t)\,dt$ for $\sigma^2(kh)$ is given by (use Lagrange theory)
\begin{equation}\label{Eqphik}
\phi_k(t)=\sqrt{2}h^{-1/2}\cos\big(\pi(t-kh)/h\big) {\bf 1}_{[kh,(k+1)h]}(t),\quad t\in [0,1].
\end{equation}
The $L^2$-orthonormal system $(\phi_k)$ for $k=0,1,\ldots,h^{-1}-1$ is now used to construct Gaussian shift observations. In ${\cal E}_2$ we obtain from \eqref{EqCovY} the observations
\begin{equation}\label{Eqyk}
y_k:=\int \phi_k(t)\,dY_t=\Big(h^2\pi^{-2}\sigma^2(kh)+\eps^2\Big)^{1/2}\zeta_k,\quad k=0,\ldots,h^{-1}-1,
\end{equation}
with independent standard normal random variables $(\zeta_k)_{k=0,\ldots,h^{-1}-1}$. Observing $(y_k)$ is clearly equivalent to observing
\begin{equation}\label{Eqzk}
z_k:=\log(y_k^2h^{-2}\pi^2)-\E[\log(\zeta_k^2)]
=\log\Big(\sigma^2(kh)+\eps^2h^{-2}\pi^2\Big)+\eta_k
\end{equation}
for $k=0,\ldots,h^{-1}-1$ with $\eta_k:=\log(\zeta_k^2)-\E[\log(\zeta_k^2)]$.

We have found a nonparametric regression model with regression function $\log(\sigma^2(\cdot)+\eps^2h^{-2}\pi^2)$ and $h^{-1}$ equidistant observations corrupted by non-Gaussian, but centered noise $(\eta_k)$ of variance 2. To ensure that the regression function does not change under the asymptotics $\eps\to 0$, we specify the block size $h=h(\eps)=h_0\eps$ with some fixed constant $h_0>0$.

It is not surprising that the nonparametric regression experiment in \eqref{Eqzk} is equivalent to a corresponding Gaussian shift experiment. Indeed, this follows readily from results by \cit{GramaNussbaum02} who in their Section 4.2 derive asymptotic equivalence already for our Gaussian scale model \eqref{Eqyk}. Note, however, that their Fisher information should be $I(\theta)=\frac12 \theta^{-2}$ and we thus have asymptotic equivalence of \eqref{Eqyk} with the Gaussian regression model
\[ w_k=\tfrac{1}{\sqrt{2}}\log(\sigma^2(kh)+h_0^{-2}\pi^2)+\gamma_k,\quad k=0,\ldots,h^{-1}-1,
\]
where $\gamma_k\sim N(0,1)$ i.i.d. Since by the classical result of \cit{BrownLow} the Gaussian regression is equivalent to the corresponding white noise experiment (note that $\log(\sigma^2(\cdot)+h_0^{-2}\pi^2)$ is also $\alpha$-H\"older continuous), we have already derived an important and far-reaching result.

\begin{theorem}\label{ThmEquivE1G1}
For $\alpha>1/2$ and $\underline\sigma^2>0$ the high frequency experiment ${\cal E}_1(\eps,\alpha,R,\underline{\sigma}^2)$ is asymptotically more informative than the Gaussian shift experiment ${\cal G}_1(\eps,\alpha,R,\underline{\sigma}^2,h_0)$ of observing
\[ dZ_t=\tfrac{1}{\sqrt{2}}\log\Big(\sigma^2(t)+h_0^{-2}\pi^2\Big)\,dt
+h_0^{1/2}\eps^{1/2}dW_t,\quad t\in [0,1].
\]
Here $h_0>0$ is an arbitrary constant and $\sigma^2\in {\cal S}(\alpha,R,\underline{\sigma}^2)$.
\end{theorem}

\begin{remark}\label{RemG1loss}
Moving the constants from the diffusion to the drift part, the experiment ${\cal G}_1$ is equivalent to observing
\begin{equation}\label{EqZtilde}
d\tilde Z_t=(2h_0)^{-1/2}\log(\sigma^2(t)+h_0^{-2}\pi^2)\,dt
+\eps^{1/2}dW_t,\quad t\in [0,1].
\end{equation}
The Gaussian shift experiment is nonlinear in $\sigma^2$ which is to be expected. Writing $\eps=\delta/\sqrt{n}$ gives us the noise level $\delta^{1/2}n^{-1/4}$ which appears in all previous work on the model ${\cal E}_0$.

To quantify the amount of information we have lost, let us study the LAN-property of the constant parametric case $\sigma^2(t)=\sigma^2>0$ in ${\cal G}_1$. We consider the local alternatives $\sigma_\eps^2=\sigma_0^2+\eps^{1/2}$ for which we obtain the Fisher information $I_{h_0}=(2h_0)^{-1}h_0^4/(\pi^2+h_0^2\sigma_0^2)^2$. Maximizing over $h_0$ yields $h_0=\sqrt{3}\pi\sigma_0^{-1}$ and the Fisher information is at most equal to \[\sup_{h_0>0}I_{h_0}=\sigma_0^{-3}3^{3/2}/(32\pi)\approx 0.0517\sigma_0^{-3}.
\]
By the LAN-result of \cit{GloterJacodI} for ${\cal E}_0$ the best value is $I(\sigma_0)=\frac18\sigma_0^{-3}$ which is clearly larger. Note, however, that the relative (normalized) efficiency is already $\frac{\sqrt{3^{3/2}/(32\pi)}}{\sqrt{1/8}}\approx 0.64$, which means that we attain about $64\%$ of the precision when working with ${\cal G}_1$ instead of ${\cal E}_0$ or ${\cal E}_1$.
\end{remark}


\section{A close sequence of simple models}\label{SecSeqSimple}

In order to decrease the information loss in ${\cal G}_1$, we now take into account higher frequencies in each block $[kh,(k+1)h]$.
In a frequency-location notation $(j,k)$ we consider for $k=0,1,\ldots,h^{-1}-1,\,j\ge 1$
\begin{equation}\label{Eqphijk}
\phi_{jk}(t)=\sqrt{2}h^{-1/2}\cos(j\pi(t-kh)/h){\bf 1}_{[kh,(k+1)h]}(t),\quad t\in [0,1].
\end{equation}
This gives the corresponding antiderivatives
\[ \Phi_{jk}(t)=\frac{\sqrt{2h}}{\pi j}\sin(j\pi(t-kh)/h){\bf 1}_{[kh,(k+1)h]}(t),\quad t\in [0,1].
\]
Not only the $(\phi_{jk})$ and $(\Phi_{jk})$ are localized on each block, also each single family of functions is orthogonal in $L^2([0,1])$. Working again on the piecewise constant experiment ${\cal E}_2$, we extract the observations
\begin{equation}\label{Eqyjk}
y_{jk}:=\int_0^1 \phi_{jk}(t)\,dY_t=\Big(h^2\pi^{-2}j^{-2}\sigma^2(kh)+\eps^2\Big)^{1/2}\zeta_{jk},
j\ge 1,\,k=0,\ldots,h^{-1}-1,
\end{equation}
with $\zeta_{jk}\sim N(0,1)$ independent over all $(j,k)$. The same transformation as before leads for each $j\ge 1$ to the regression model for $k=0,\ldots,h^{-1}-1$
\begin{equation}\label{Eqzjk} z_{jk}:=\log(y_{jk}^2)-\log(h^2\pi^{-2}j^{-2})-\E[\log(\zeta_{jk}^2)]
=\log(\sigma^2(t)+\eps^2h^{-2}\pi^2j^2)+\eta_{jk}.
\end{equation}
Applying the asymptotic equivalence result by \cit{GramaNussbaum02} for each independent level $j$ separately, we immediately generalize Theorem \ref{ThmEquivE1G1}.

\begin{theorem}\label{ThmE1G2}
For $\alpha>1/2$ and $\underline{\sigma}^2>0$ the high frequency experiment ${\cal E}_1(\eps,\alpha,R,\underline{\sigma}^2)$ is asymptotically more informative than the combined experiment ${\cal G}_2(\eps,\alpha,R,\underline{\sigma}^2,h_0,J)$ of independent Gaussian shifts
\[ dZ^j_t=\tfrac{1}{\sqrt{2}}\log(\sigma^2(t)+h_0^{-2}\pi^2j^2)\,dt
+h_0^{1/2}\eps^{1/2}dW^j_t,\quad t\in [0,1],\,j=1,\ldots,J,
\]
with independent Brownian motions  $(W^j)_{j=1,\ldots,J}$ and $\sigma^2\in {\cal S}(\alpha,R,\underline{\sigma}^2)$.
The constants $h_0>0$ and $J\in\N$ are arbitrary, but fixed.
\end{theorem}

\begin{remark}
Let us again study the LAN-property of the constant parametric case $\sigma^2(t)=\sigma^2>0$ for the local alternatives $\sigma^2_\eps=\sigma_0^2+\eps^{1/2}$. We obtain the Fisher information
\[I_{h_0,J}=\sum_{j=1}^J(2h_0)^{-1}h_0^4(\pi^2j^2+h_0^2\sigma_0^2)^{-2}
=\sum_{j=1}^J\frac{h_0^{-1}}{2(\pi^2(jh_0^{-1})^2+\sigma_0^2)^2}.\]
In the limit $J\to\infty$ and $h_0\to\infty$ we obtain by Riemann sum approximation
\[ \lim_{h_0\to\infty}\lim_{J\to\infty} I_{h_0,J}=\int_0^\infty \frac{dx}{2(\pi^2x^2+\sigma_0^2)^2}
=\frac{1}{8\sigma_0^3}.
\]
This is exactly the optimal Fisher information, obtained by \cit{GloterJacodI} in this case. Note, however, that it is not at all obvious that we may let $J,h_0\to\infty$, in the asymptotic equivalence result. Moreover, in our theory the restriction $h^\alpha=o(\eps^{1/2})$ is necessary, which translates into $h_0=o(\eps^{(1-2\alpha)/2\alpha})$. Still, the positive aspect is that we can come as close as we wish to an asymptotically almost equivalent, but much simpler model.
\end{remark}

\section{Localisation}\label{SecLoc}

We know from standard regression theory \cite{Stone82} that in the experiment ${\cal G}_1$ we can estimate $\sigma^2\in C^\alpha$ in sup-norm with rate $(\eps\log(\eps^{-1}))^{\alpha/(2\alpha+1)}$, using that the log-function is a $C^\infty$-diffeomorphism for arguments bounded away from zero and infinity. Since ${\cal E}_1$ is for $\alpha>1/2$ asymptotically more informative than ${\cal G}_1$, we can therefore localize $\sigma^2$ in a neighbourhood of some $\sigma_0^2$. Using the local coordinate $s^2$ in $\sigma^2=\sigma_0^2+v_\eps s^2$ for $v_\eps\to 0$ we define a localized experiment, cf. \cit{Nussbaum}.

\begin{definition}
Let ${\cal E}_{i,loc}={\cal E}_{i,loc}(\sigma_0,\eps,\alpha,R,\underline{\sigma}^2)$ for $\sigma_0\in {\cal S}(\alpha,R,\underline{\sigma}^2)$ be the statistical subexperiment obtained from ${\cal E}_i(\eps,\alpha,R,\underline{\sigma}^2)$ by restricting to the parameters $\sigma^2=\sigma_0^2+v_\eps s^2$ with $v_\eps=\eps^{\alpha/(2\alpha+1)}\log(\eps^{-1})$ and unknown $s^2\in C^\alpha(R)$.
\end{definition}

We shall consider the observations $(y_{jk})$ in \eqref{Eqyjk} derived from ${\cal E}_{2,loc}$ and multiplied by $\pi j/h$. The model is then a generalized nonparametric regression family in the sense of \cit{GramaNussbaum02}. On the sequence space $({\cal X},{\cal F})=(\R^{\N},{\mathfrak B}^{\otimes\N})$ we consider for $\theta\in\Theta=[\underline\sigma^2,R]$ the Gaussian product measure
\begin{equation}\label{EqPtheta}
 \PP_\theta=\bigotimes_{j\ge 1} N\big(0,\theta+h_0^{-2}\pi^2j^2\big).
\end{equation}
The parameter $\theta$ plays the role of $\sigma^2(kh)$ for each $k$. By independence and the result for the one-dimensional Gaussian scale model, the Fisher information for $\theta$ is given by
\begin{equation}\label{EqItheta} I(\theta):=\sum_{j\ge 1} \frac{1}{2(\theta+h_0^{-2}\pi^2j^2)^2}
=\frac{h_0}{8\theta^{3/2}}\Big(\frac{1+4\theta^{1/2} h_0 e^{-2\theta^{1/2} h_0}-e^{-4\theta^{1/2} h_0}}{(1-e^{-2\theta^{1/2} h_0})^2}
-\frac{2}{\theta^{1/2} h_0}\Big),
\end{equation}
where the series is evaluated in Section \ref{SecSeries} using Fourier analysis. Since we shall later let $h_0$ tend to infinity, an essential point is the asymptotics $I(\theta)\thicksim h_0$.

We split our observation design $\{kh\,|\,k=0,\ldots,h^{-1}\}$ into blocks $A_m=\{kh\,|\,k=(m-1)\ell,\ldots,m\ell-1\}$, $m=1,\ldots,(\ell h)^{-1}$, of length $\ell$ such that the radius $v_\eps$ of our nonparametric local neighbourhood has the order of the {\it parametric}  noise level $(I(\theta)\ell)^{-1/2}$ in each block:
\[ v_\eps\thicksim (I(\theta)\ell)^{-1/2}\Rightarrow \ell\thicksim h_0^{-1}v_\eps^{-2}.
\]

For later convenience we consider odd and even  indices $k$ separately, assuming that $h^{-1}$ and $\ell$ are even integers.
This way, for each block $m$ observing $(y_{jk}\pi j/h)$ for $j\ge 1$ and $k\in A_m$, $k$ odd respectively $k$ even, can be modeled by the experiments
\begin{align}\label{EqE3m}
{\cal E}_{3,m}^{odd}&=\Big({\cal X}^{\ell/2},{\cal F}^{\otimes\ell/2},\Big(\bigotimes_{k\in A_m\text{ odd}}\PP_{\sigma_0^2(k/n)+v_\eps s^2(k/n)}\Big)_{s^2\in C_\alpha(R)}\Big),\\
{\cal E}_{3,m}^{even}&=\Big({\cal X}^{\ell/2},{\cal F}^{\otimes\ell/2},\Big(\bigotimes_{k\in A_m\text{ even}}\PP_{\sigma_0^2(k/n)+v_\eps s^2(k/n)}\Big)_{s^2\in C_\alpha(R)}\Big),
\end{align}
where all parameters are the same as for ${\cal E}_{2,loc}$. Using the nonparametric local asymptotic theory developed by \cit{GramaNussbaum02} and the independence of the experiments $({\cal E}_{3,m}^{odd})_m$ (resp. $({\cal E}_{3,m}^{even})_m$), we are able to prove in Section \ref{SecProofPropE1G3loc} the following asymptotic equivalence.

\begin{proposition}\label{PropE1G3loc}
Assume $\alpha>1/2$, $\underline\sigma^2>0$ and $h_0\thicksim\eps^{-p}$ with $p\in (0,1-(2\alpha)^{-1})$ such that $(2h)^{-1}\in\N$. Then observing $\{y_{j,2k+1}\,|\,j\ge 1, k=0,\ldots,(2h)^{-1}-1\}$ in experiment ${\cal E}_{2,loc}$ is asymptotically equivalent to the local Gaussian shift experiment ${\cal G}_{3,loc}$ of observing
\begin{equation}\label{EqG3loc}
dY_t=\frac{1}{\sqrt{8}\sigma_0^{3/2}(t)}\Big(1
-\frac{2}{\sigma_0(t)h_0}\Big)^{1/2}v_\eps s^2(t)\,dt+ (2\eps)^{1/2}dW_t,\quad t\in [0,1],
\end{equation}
where the unknown $s^2$ and all parameters are the same as in ${\cal E}_{2,loc}$. The Le Cam distance tends to zero uniformly over the center of localisation $\sigma_0^2\in{\cal S}(\alpha,R,\underline{\sigma}^2)$.

The same asymptotic equivalence result holds true for observing $\{y_{j,2k}\,|\,j\ge 1, k=0,\ldots,(2h)^{-1}-1\}$ in experiment ${\cal E}_{2,loc}$.
\end{proposition}

Note that in this model, combining even and odd indices $k$, we can already infer the LAN-result by \cit{GloterJacodI}, but we still face a second order term of order $h_0^{-1}v_\eps$ in the drift. This term is asymptotically negligible only if it is of smaller order than the noise level $\eps^{1/2}$. To be able to choose $h_0$ sufficiently large, we have to require a larger H\"older smoothness of the volatility.

\begin{corollary}\label{CorE1G4loc}
Assume $\alpha>\frac{1+\sqrt{17}}{8}\approx 0.64$, $\underline\sigma^2>0$ and $(2h)^{-1}\in\N$. Then observing $\{y_{j,2k+1}\,|\,j\ge 1, k=0,\ldots,(2h)^{-1}-1\}$ in experiment ${\cal E}_{2,loc}$ is asymptotically equivalent to the local Gaussian shift experiment ${\cal G}_{4,loc}$ of observing
\begin{equation}\label{EqG4loc}
dY_t=\frac{1}{\sqrt{8}\sigma_0^{3/2}(t)}v_\eps s^2(t)\,dt+ (2\eps)^{1/2}dW_t,\quad t\in [0,1],
\end{equation}
where the unknown $s^2$ and all parameters are the same as in ${\cal E}_{2,loc}$. The Le Cam distance tends to zero uniformly over the center of localisation $\sigma_0^2\in{\cal S}(\alpha,R,\underline{\sigma}^2)$.

The same asymptotic equivalence result holds true for observing $\{y_{j,2k}\,|\,j\ge 1, k=0,\ldots,(2h)^{-1}-1\}$ in experiment ${\cal E}_{2,loc}$.
\end{corollary}

\begin{proof}
For $\alpha>\frac{1+\sqrt{17}}{8}$ the choice of $h_0=\eps^{-p}$ for some $p\in (\frac{1}{4\alpha+2},\frac{2\alpha-1}{2\alpha})$ is possible and ensures that $h^{\alpha}=o(\eps^{1/2})$ holds as well as $h_0^{-2}=o(v_\eps^{-2}\eps)$. Therefore the Kullback-Leibler divergence between the observations in ${\cal G}_3^{loc}$ and in ${\cal G}_4^{loc}$ evaluates by the Cameron-Martin (or Girsanov) formula to
\[ \eps^{-1}\int_0^1 \frac{1}{8\sigma_0^{3}(t)}\Big(\Big(1
-\frac{2}{\sigma_0(t)h_0}\Big)^{1/2}-1\Big)^2v_\eps^2 s^4(t)\,dt\lesssim \eps^{-1}h_0^2v_\eps^2.
\]
Consequently, the Kullback-Leibler and thus also the total variation distance tends to zero.
\end{proof}

In a last step we find local experiments ${\cal G}_{5,loc}$, which are asymptotically equivalent to ${\cal G}_{4,loc}$ and do not depend on the center of localisation $\sigma_0^2$. To this end we use a variance-stabilizing transform, based on the Taylor expansion
\[
\sqrt{2}x^{1/4}=\sqrt{2}x_0^{1/4}+\tfrac{1}{\sqrt{8}}x_0^{-3/4}(x-x_0)+O((x-x_0)^2)
\]
which holds uniformly over $x,x_0$ on any compact subset of $(0,\infty)$. Inserting $x=\sigma^2(t)=\sigma_0^2(t)+v_\eps s^2(t)$ and $x_0=\sigma_0^2$ from our local model, we obtain
\begin{equation}\label{EqIthetaTaylor}
\sqrt{2\sigma(t)}=\sqrt{2\sigma_0(t)}+\tfrac{1}{\sqrt{8}}\sigma_0^{-3/2}(t)v_\eps s^2(t)+O(v_\eps^2).
\end{equation}

Since $v_\eps^2=o(\eps^{1/2})$ holds for $\alpha>1/2$, we can add the uninformative signal $\sqrt{2}\sigma_0^{1/2}(t)$ to $Y$ in ${\cal G}_{4,loc}$, replace the drift by
$\sqrt{2}\sigma^{1/2}(t)$ and still keep convergence of the total variation distance, compare the preceding proof. Consequently, from Corollary \ref{CorE1G4loc} we obtain the following result.

\begin{corollary}\label{CorE1G5loc}
Assume $\alpha>\frac{1+\sqrt{17}}{8}\approx 0.64$, $\underline\sigma^2>0$ and $(2h)^{-1}\in\N$.  Then observing $\{y_{j,2k+1}\,|\,j\ge 1, k=0,\ldots,(2h)^{-1}-1\}$ in the experiment ${\cal E}_{2,loc}$ is asymptotically equivalent to the local Gaussian shift experiment ${\cal G}_{5,loc}$ of observing
\begin{equation}\label{EqG5loc}
dY_t=\sqrt{2\sigma(t)}\,dt+ (2\eps)^{1/2}\,dW_t,\quad t\in [0,1],
\end{equation}
where the unknown is $\sigma^2=\sigma_0^2+v_\eps s^2$ and all parameters are the same as in ${\cal E}_{2,loc}$. The Le Cam distance tends to zero uniformly over the center of localisation $\sigma_0^2\in{\cal S}(\alpha,R,\underline{\sigma}^2)$.

The same asymptotic equivalence result holds true for observing $\{y_{j,2k}\,|\,j\ge 1, k=0,\ldots,(2h)^{-1}-1\}$ in experiment ${\cal E}_{2,loc}$.
\end{corollary}

\section{Globalisation}\label{SecGlob}

The globalisation now basically follows the usual route, first established by \cit{Nussbaum}. Essential for us is to show that observing $(y_{jk})$ for $j\ge 1$ is asymptotically sufficient in ${\cal E}_2$. Then we can split the white noise observation experiment ${\cal E}_2$ into two independent sub-experiments obtained from $(y_{jk})$ for $k$ odd and $k$ even, respectively. Usually, a white noise experiment can be split into two independent subexperiments with the same drift and an increase by $\sqrt{2}$ in the noise level. Here, however, this does not work since the two diffusions in the {\it random} drift remain the same and thus independence fails.

Let us introduce the $L^2$-normalized step functions
\begin{align*}
\phi_{0,k}(t) &:=(2h)^{-1/2}\big({\bf 1}_{[(k-1)h,kh]}(t)-{\bf 1}_{[kh,(k+1)h]}(t)\big),\quad k=1,\ldots,h^{-1}-1,\\
\phi_{0,0}(t) &:=h^{-1/2}{\bf 1}_{[0,h]}(t).
\end{align*}
We obtain a normalized complete basis $(\phi_{jk})_{j\ge 0,0\le k\le h^{-1}-1}$ of $L^2([0,1])$ such that observing $Y$ in experiment ${\cal E}_2$ is equivalent to observing
\[ y_{jk}:=\int_0^1\phi_{jk}(t)\,dY_t,\quad j\ge 0,\,k=0,\ldots,h^{-1}-1.
\]
Calculating the Fourier series, we can express the tent function $\Phi_{0,k}$ with $\Phi_{0,k}'=\phi_{0,k}$ and $\Phi_{0,k}(1)=0$ as an $L^2$-convergent series over the dilated sine functions $\Phi_{jk}$ and $\Phi_{j,k-1}$, $j\ge 1$:
\begin{equation}\label{EqPhi0k}
\Phi_{0,k}(t)=\sum_{j\ge 1}(-1)^{j+1}\Phi_{j,k-1}(t)+\sum_{j\ge 1}\Phi_{jk}(t),\quad k=1,\ldots, h^{-1}-1.
\end{equation}
We also have $\Phi_{0,0}(t)=2\sum_{j\ge 1}\Phi_{j,0}(t)$.
By partial integration, this implies (with $L^2$-convergence)
\begin{align*}
&\beta_{0,k} :=\scapro{\phi_{0,k}}{X}
=-\int_0^1 \Phi_{0,k}(t)\,dX(t)
=\sum_{j\ge 1}(-1)^{j+1}\beta_{j,k-1}+\sum_{j\ge 1}\beta_{jk},\\
&\text{where }\beta_{jk}:=\scapro{\phi_{jk}}{X}
\end{align*}
for $k\ge 1$ and similarly $\beta_{0,0}=2\sum_{j\ge 1}\beta_{j,0}$. This means that the signal $\beta_{0,k}$ in $y_{0,k}$ can be perfectly reconstructed from the signals in the $y_{j,k-1}$, $y_{jk}$.
For jointly Gaussian random variables we obtain the conditional law in ${\cal E}_2$
\[ {\cal L}(\beta_{jk}\,|\,y_{jk})
=N\Big(\frac{\Var(\beta_{jk})}{\Var(y_{jk})}y_{jk},
\frac{\eps^2\Var(\beta_{jk})}{\Var(y_{jk})}\Big)
\]

Given the results by \cit{Stone82} and our less informative Gaussian shift experiment ${\cal G}_1$ for $\alpha>1/2$, $\underline{\sigma}^2>0$, there is an estimator $\hat\sigma^2_\eps$ based on $(y_{1,k})_k$ in ${\cal E}_2$ with
\begin{equation}\label{EqConsistentEst}
 \lim_{\eps\to 0}\inf_{\sigma^2\in{\cal S}} \PP_{\sigma^2,\eps}(\norm{\hat\sigma^2_\eps-\sigma^2}_\infty\le Rv_\eps)=1,
\end{equation}
where $v_\eps=\eps^{\alpha/(2\alpha+1)}\log(\eps^{-1})$ as in the definitions of the localized experiments.

We can thus generate independent $N(0,1)$-distributed random variables $\rho_{jk}$ to construct from $(y_{jk})_{j\ge 1,k}$
\[ \tilde\beta_{jk}:=\frac{\Var_\eps(\beta_{jk})}{\Var_\eps(y_{jk})}y_{jk}
+\frac{\eps\Var_\eps(\beta_{jk})^{1/2}}{\Var_\eps(y_{jk})^{1/2}}\rho_{jk},
\]
where the variance $\Var_\eps$ is the expression for $\Var$ where all unknown values $\sigma^2(kh)$ are replaced by the estimated values $\hat\sigma^2_\eps(kh)$. From this we can generate artificial observations $(\tilde y_{0,k})$ such that the conditional law
${\cal L}((\tilde y_{0,k})_k\,|\,(\tilde\beta_{j,k})_k)$ coincides with ${\cal L}(( y_{0,k})_k\,|\,(\beta_{0,k})_k)$, which is just a multivariate normal law with mean zero and tri-diagonal covariance matrix $\eps^2(\scapro{\phi_{0,k}}{\phi_{0,k'}})_{k,k'}$.

In Section \ref{ProofPropE1locyjk} we shall prove that the Hellinger distance between the families of centered Gaussian random variables ${\cal Y}:=\{y_{jk}\,|\,j\ge 0,\,k=0,\ldots,h^{-1}-1\}$ and $\tilde{\cal Y}:=\{\tilde y_{0,k}\,|\,k=0,\ldots,h^{-1}-1\}\cup\{y_{jk}\,|\,j\ge 1,\,k=0,\ldots,h^{-1}-1\}$
tends to zero, provided $h_0^{-1}v_\eps^2=o(\eps)$, which is possible when $\alpha>\frac{1+\sqrt{5}}{4}$ with the choice $h_0=\eps^{-p}$ for some $p\in (\frac{1}{2\alpha+1},\frac{2\alpha-1}{2\alpha})$.

\begin{proposition}\label{PropE1locyjk}
Assume $\alpha>\frac{1+\sqrt{5}}{4}\approx 0.81$, $\underline\sigma^2>0$ and $h^{-1}$ an even integer. Then the experiment ${\cal E}_2$ is asymptotically equivalent to the product experiment ${\cal E}_{2,odd}\otimes{\cal E}_{2,even}$ where ${\cal E}_{2,odd}$ is obtained from the observations $\{y_{j,2k+1}\,|\,j\ge 1, k=0,\ldots,(2h)^{-1}-1\}$ and ${\cal E}_{2,even}$ from the observations $\{y_{j,2k}\,|\,j\ge 1, k=0,\ldots,(2h)^{-1}-1\}$ in experiment ${\cal E}_2$.
\end{proposition}

This key result permits to globalize the local result. In the sequel we always assume $\alpha>\frac{1+\sqrt{5}}{4}$ and $\underline{\sigma}^2>0$. We start with the asymptotic equivalence between ${\cal E}_2$ and ${\cal E}_{2,odd}\otimes{\cal E}_{2,even}$. Using again an estimator $\hat\sigma^2_\eps$ in ${\cal E}_{2,odd}$ satisfying \eqref{EqConsistentEst} we can localize the second factor ${\cal E}_{2,even}$ around $\hat\sigma_\eps^2$ and therefore by Corollary \ref{CorE1G5loc} replace it by experiment ${\cal G}_{5,loc}$, see Theorem 3.2 in \cit{Nussbaum} for a formal proof. Since ${\cal G}_{5,loc}$ does not depend on the center $\hat\sigma_\eps^2$, we conclude that ${\cal E}_2$ is asymptotically equivalent to the product experiment ${\cal E}_{2,odd}\otimes{\cal G}_5$ where ${\cal G}_5$ has the same parameters as ${\cal E}_2$ and is given by observing $Y$ in \eqref{EqG5loc}. Now we use an estimator $\hat\sigma^2_\eps$ in ${\cal G}_5$ satisfying \eqref{EqConsistentEst}, whose existence is ensured by \cit{Stone82}, to localize ${\cal E}_{2,odd}$.  Corollary \ref{CorE1G5loc} then allows again to replace the localized ${\cal E}_{2,odd}$-experiment by ${\cal G}_5$ such that ${\cal E}_2$ is asymptotically equivalent to the product experiment ${\cal G}_5\otimes{\cal G}_5$. Finally, taking the mean of the independent observations \eqref{EqG5loc} in both factors, which is a sufficient statistics, (or, abstractly, due to identical likelihood processes) we see that ${\cal G}_5\otimes {\cal G}_5$ is equivalent to the experiment ${\cal G}_0$ of observing
$dY_t=\sqrt{2\sigma(t)}\,dt+ \sqrt{\eps}\,dW_t$, $t\in [0,1]$.
Our final result then follows from the asymptotic equivalence between ${\cal E}_0$ and ${\cal E}_1$ as well as between ${\cal E}_1$ and ${\cal E}_2$.

\begin{theorem}\label{ThmE0G0}
Assume $\alpha>\frac{1+\sqrt{5}}{4}\approx 0.81$ and $\delta,\underline\sigma^2,R>0$. Then the regression experiment ${\cal E}_0(n,\delta,\alpha,R,\underline{\sigma}^2)$ is for $n\to\infty$ asymptotically equivalent to the Gaussian shift experiment ${\cal G}_0(\delta n^{-1/2},\alpha,R,\underline{\sigma}^2)$ of observing
\begin{equation}\label{EqG0}
dY_t=\sqrt{2\sigma(t)}\,dt+ \delta^{1/2}n^{-1/4}\,dW_t,\quad t\in [0,1],
\end{equation}
for $\sigma^2\in {\cal S}(\alpha,R,\underline{\sigma}^2)$.
\end{theorem}

\section{Discussion}\label{SecDisc}

Our results show that inference for the volatility in the high-frequency observation model under microstructure noise ${\cal E}_0$ is asymptotically as difficult as in the well understood Gaussian shift model ${\cal G}_0$.  Remark that the constructions in \cit{GloterJacodI}, \cit{GloterJacodII} rely on preliminary estimators at the boundary of suitable blocks, while we require $\supp\Phi_{jk}=[kh,(k+1)h]$ to obtain independence among blocks. In this context Proposition \ref{PropE1locyjk} shows asymptotic sufficiency of observing only the pinned process $X_t-\frac{(k+1)h-t}{h}X_{kh}-\frac{t-kh}{h}X_{(k+1)h}$, $t\in [kh,(k+1)h]$, on each block due to $\int (\alpha t+\beta)\phi_{jk}(t)\,dt=0$ for $j\ge 1$, $\alpha,\beta\in\R$. Naturally, the $(\Phi_{jk})_{j\ge 1}$ form exactly the eigenfunctions of the covariance operator of the Brownian bridge.

It is interesting to note that both, model ${\cal E}_0$ and model ${\cal G}_0$, are homogeneous in the sense that factors from the noise (i.e. the $dW_t$-term) can be moved to the drift term and vice versa such that for example high volatility can counterbalance a high noise level $\delta$ or a large observation distance $1/n$. Another phenomenon is that observing ${\cal E}_0$ $m$-times independently, in particular with different realisations of the process $X$, is asymptotically as informative as observing ${\cal E}_0$ with $m^2$ as many observations: both experiments are asymptotically equivalent to $dY_t=\sqrt{2\sigma(t)}dt+m^{1/2}\delta^{1/2}n^{-1/4}dW_t$. Similarly, by rescaling we can treat observations on intervals $[0,T]$ with $T>0$ fixed: Observing $Y_i=X_{iT/n}+\eps_i$, $i=1,\ldots,n$, in ${\cal E}_0$ with $X_t=\int_0^t\sigma(s)\,dB_s$, $t\in [0,T]$, is under the same conditions asymptotically equivalent to observing
\[dY_u=\sqrt{2\sigma(Tu)}\,du+\delta^{1/2}T^{-1/4}n^{-1/4}\,dW_u,\quad u\in [0,1],
\]
or equivalently, 
\[d\tilde Y_v= \sqrt{2\sigma(v)}\,du+\delta^{1/2}T^{1/4}n^{-1/4}\,dW_v,\quad v\in [0,T].
\]
Concerning the various restrictions on the smoothness $\alpha$ of the volatility $\sigma^2$, one might wonder whether the critical index is $\alpha=1/2$ in view of the classical asymptotic equivalence results \cite{BrownLow,Nussbaum}. In our approach,  we still face the second order term in \eqref{EqG3loc} and using the localized results, a much easier globalisation yields for $\alpha>1/2$ only that ${\cal E}_0$ is asymptotically not less informative than observing
\[dY_t=F(\sigma^2(t))\,dt+ \delta^{1/2}n^{-1/4}dW_t,\quad t\in [0,1],
\]
with $F(x)=\int_1^x (y^{1/2}-2h_0^{-1})^{1/2}y^{-1}dy/\sqrt{8}$, which includes a small, but non-negligible second-order term since $h_0$ cannot tend to infinity too quickly.

On the other hand, it is quite easy to see that for $\alpha\le 1/4$ asymptotic equivalence fails. In the regression model ${\cal E}_0$ with $n$ observations we cannot distinguish between $X_n(t)=\int_0^t\sigma_n(t)\,dB_t$ with $\sigma_n^2(t)=1+n^{-1/4}\cos(\pi nt)$, $\norm{\sigma_n^2}_{C^{1/4}}=2+n^{-1/4}$, and standard Brownian motion ($\sigma^2=1$) since $X_n(i/n)-X_n((i-1)/n)\sim N(0,1/n)$ i.i.d. holds. On the other hand, we have $\int_0^1(\sqrt{2\sigma_n(t)}-\sqrt{2})^2\,dt\thicksim n^{-1/2}$, which shows that the signal to noise ratio in the Gaussian shift ${\cal G}_0$ is of order $1$ and a Neyman-Pearson test between $\sigma_n^2$ and $1$ can distinguish both signals with a positive probability. This different behaviour for testing in ${\cal E}_0$ and ${\cal G}_0$ implies that both models cannot be asymptotically equivalent for $\alpha=1/4$. Note that \cit{GloterJacodI} merely require $\alpha\ge 1/4$ for their LAN-result, but our counterexample is excluded by their parametric setting. In conclusion, the behaviour in the zone $\alpha\in (1/4,(1+\sqrt{5})/4]$ remains unexplored.

\section{Applications}\label{SecAppl}

Let us first consider the nonparametric problem of estimating the spot volatility $\sigma^2(t)$. From our asymptotic equivalence result in Theorem \ref{ThmE0G0} we can deduce, at least for bounded loss functions, the usual nonparametric minimax rates, but with the number $n$ of observations replaced by $\sqrt{n}$ provided $\sigma^2\in C^\alpha$ for $\alpha>(1+\sqrt{5})/4$ as the mapping $\sqrt{\sigma(t)}\mapsto \sigma^2(t)$ is a $C^\infty$-diffeomorphism for volatilities $\sigma^2$ bounded away from zero. Since the results so far obtained only deal with rate results, it is even simpler to use our less informative model ${\cal G}_1$ or more concretely the observations $(y_k)$ in \eqref{Eqyk} which are
independent in ${\cal E}_2$, centered and of variance $h^2\pi^{-2}\sigma^2(kh)+\eps^2$. With $h=\eps$ a local (kernel or wavelet) averaging over $\eps^{-2}\pi^2y_k^2-\pi^2$ therefore  yields rate-optimal estimators for classical pointwise or $L^p$-type loss functions.

For later use we choose $h=\eps$ in ${\cal E}_2$ and propose the simple estimator
\[ \hat\sigma^2_b(t):=\frac{\eps}{2b}\sum_{k:\abs{k\eps-t}\le b} (\eps^{-2}\pi^2y_k^2-\pi^2)
\]
for some bandwidth $b>0$. Since $\zeta_k^2$ is $\chi^2(1)$-distributed, it is standard \cite{Stone82} to show that with the choice $b\thicksim (\eps\log(\eps^{-1}))^{1/(2\alpha+1)}$ we have the sup-norm risk bound
\[ \E[\norm{\hat\sigma^2_b-\sigma^2}^2_\infty]\lesssim (\eps\log(\eps^{-1}))^{2\alpha/(2\alpha+1)},
\]
especially we shall need that $\hat\sigma^2_b$ is consistent in sup-norm loss.

In terms of the regression experiment ${\cal E}_0$ we work (in an asymptotically equivalent way) with the linear interpolation $\hat Y'$ of the observations $(Y_i)$, see the proof of Theorem \ref{ThmRegression}. By partial integration we can thus take for any $j,k$
\begin{equation}\label{Eqyjk0}
y_{jk}^0:=-\int_0^1\Phi_{jk}(t)\hat Y''(t)\,dt=\sum_{i=1}^n \Big(-\int_{(i-1)/n}^{i/n}\Phi_{jk}(t)\,dt\Big) (Y_i-Y_{i-1}),
\end{equation}
setting $Y_0:=0$. Note that we have the uniform approximation $y_{jk}^0= \frac{-1}{n}\sum_{i=1}^n\Phi_{jk}(i/n)(Y_i-Y_{i-1}) + O(h^{-1/2}n^{-1})$ due to $\norm{\phi_{jk}}_\infty\le (2h)^{-1/2}$.
We see the relationship with the pre-averaging approach. The idea of using disjoint averages is present in \cit{PodolskijVetter}, where in our terminology Haar functions are used as $\Phi_k$. They were aware of the fact that discretized sine functions would slightly increase the Fisher information (personal communication, see also their discussion after Corollary 2), but they have not used higher frequencies.

Since we use the concrete coupling by linear interpolation to define $y_{jk}^0$ in ${\cal E}_0$ and since convergence in total variation is stronger than weak convergence, all asymptotics for probabilities and weak convergence results for functionals $F((y_{jk})_{jk})$ in ${\cal E}_2$ remain true for $F((y_{jk}^0)_{jk})$ in ${\cal E}_0$, uniformly over the parameter class. The formal argument for the latter is that whenever $\norm{\PP_n-\QQ_n}_{TV}\to 0$ and $\PP_n^{X_n}\to \PP$ weakly for some random variables $X_n$ we have for all bounded and continuous $g$
\[ \E_{\QQ_n}[g(X_n)]=\E_{\PP_n}[g(X_n)]+O(\norm{g}_\infty\norm{\PP_n-\QQ_n}_{TV})
\xrightarrow{n\to\infty} \E_{\PP}[g(X)].
\]
Thus, for $\alpha>1/2$, $\underline\sigma^2>0$ and $b\thicksim (n^{-1/2}\log n)^{-1/(2\alpha+1)}$ the estimator
\begin{equation}\label{EqSigmaTilde} \tilde\sigma^2_n(t):=\frac{\delta}{2b\sqrt{n}}\sum_{k:\abs{kn^{-1/2}-t}\le b} (n\delta^{-2}\pi^2(y_k^0)^2-\pi^2)
\end{equation}
satisfies in the regression experiment ${\cal E}_0$
\begin{equation}\label{EqConsistentEstE0}
 \lim_{n\to \infty}\inf_{\sigma^2\in{\cal S}(\alpha,R,\underline{\sigma}^2)} \PP_{\sigma^2,n}(n^{\alpha/(4\alpha+2)}(\log n)^{-1} \norm{\tilde\sigma_n^2-\sigma^2}_\infty\le R)=1.
\end{equation}


The asymptotic equivalence can be applied to construct estimators for the integrated volatility $\int_0^1\sigma^2(t)dt$ or more generally $p$-th order integrals $\int_0^1\sigma^p(t)dt$ using the approach developed by \cit{IbrKhas2} for white noise models like ${\cal G}_0$. In our notation their Theorem 7.1 yields an estimator $\hat\theta_{p,n}$ of $\int_0^1\sigma^p(t)dt$  in ${\cal G}_0$ such that
\[ \E_{\sigma^2}\Big[\Big(\hat\theta_{p,n}-\int_0^1\sigma^p(t)\,dt-\delta^{1/2}n^{-1/4}
\sqrt{2}p\int_0^1 \sigma^{p-1/2}(t)\,dW_t\Big)^2\Big]=o(n^{-1/2})
\]
holds uniformly over $\sigma^2\in {\cal S}(\alpha,R,\underline{\sigma}^2)$ for any $\alpha,R,\underline{\sigma}^2>0$ since the functional $\sqrt{\sigma(\cdot)}\mapsto \int_0^1 \sigma^p(t)dt$ is smooth on $L^2$. A LAN-result shows that asymptotic normality with rate $n^{-1/4}$ and variance $\delta 2p^2\int_0^1\sigma^{2p-1}(t)\,dt$ is minimax optimal. Specializing to the case $p=2$ for integrated volatility, the asymptotic variance is $8\delta\int_0^1 \sigma^3(t)\,dt$. It should be stressed here that the existing estimation procedures for integrated volatility are globally sub-optimal for our idealized model in the sense that their asymptotic variances involve the integrated quarticity $\int_0^1 \sigma^4(t)\,dt$ which can at most yield optimal variance for constant values of $\sigma^2$, because otherwise $\int_0^1\sigma^4(t)\,dt>\big(\int_0^1\sigma^3(t)\,dt\big)^{4/3}$
follows from Jensen's inequality. The fundamental reason is that all these estimators are based on quadratic forms of the increments depending on
global tuning parameters, whereas optimizing weights locally permits to attain the above efficiency bound as we shall see.

Instead of following these more abstract approaches, we use our analysis to construct a simple estimator of the integrated volatility with optimal asymptotic variance. First we use the statistics $(y_{jk})$ in ${\cal E}_2$ and then transfer the results to ${\cal E}_0$ using $(y_{jk}^0)$ from \eqref{Eqyjk0}.

On each block $k$ we dispose in ${\cal E}_2$ of independent $N(0,h^2j^{-2}\pi^{-2}\sigma^2(kh)+\eps^2)$-observations $y_{jk}$ for $j\ge 1$. A maximum-likelihood estimator $\hat\sigma^2(kh)$ in this exponential family satisfies the estimating equation
\begin{align}
\hat\sigma^2(kh)&=\sum_{j\ge 1} w_{jk}(\hat\sigma^2)h^{-2}j^2\pi^2(y_{jk}^2-\eps^2),\\
\text{where } w_{jk}(\sigma^2)&:=\frac{(\sigma^2(kh)+ h_0^{-2}\pi^2j^2)^{-2}} {\sum_{l\ge 1}(\sigma^2(kh)+ h_0^{-2}\pi^2l^2)^{-2}}.
\end{align}
This can be solved numerically, yet it is a non-convex problem (personal communication by J. Schmidt-Hieber). Classical MLE-theory, however, asserts for fixed $h$, $k$ and consistent initial estimator $\tilde\sigma_n^2(kh)$ that only one Newton step suffices to ensure asymptotic efficiency. Because of $h\to 0$ this immediate argument does not apply here, but still gives rise to the estimator
\[\widehat{IV}_\eps:=\sum_{k=0}^{h^{-1}-1} h\sum_{j\ge 1} w_{jk}(\tilde\sigma_n^2)h^{-2}j^2\pi^2(y_{jk}^2-\eps^2)
\]
of the integrated volatility $IV:=\int_0^1\sigma^2(t)\,dt$. Assuming the $L^\infty$-consistency $\norm{\tilde\sigma_n^2-\sigma^2}_\infty\to 0$ in probability for the initial estimator, we assert in ${\cal E}_2$ the efficiency result
\[ \eps^{-1/2}(\widehat{IV}_\eps-IV)\xrightarrow{\cal L} N\Big(0,8\int_0^1\sigma^3(t)\,dt\Big).
\]
To prove this, it suffices by Slutsky's lemma to show
\begin{align}
\eps^{-1/2}\sum_{k=0}^{h^{-1}-1} h\sum_{j\ge 1} w_{jk}(\sigma^2)h^{-2}j^2\pi^2(y_{jk}^2-\eps^2)&\xrightarrow{\cal L} N\Big(0,8\int_0^1\sigma^3(t)\,dt\Big),\label{EqAssert1}\\
\sup_{jk}\abs{w_{jk}(\tilde\sigma_n^2)-w_{jk}(\sigma^2)}&\lesssim w_{jk}(\sigma^2)\norm{\tilde\sigma_n^2-\sigma^2}_\infty.\label{EqAssert2}
\end{align}
The second assertion \eqref{EqAssert2} follows from inserting the Lipschitz property that $W(x):=(x+ h_0^{-2}\pi^2j^2)^{-2}$ satisfies $\abs{W'(x)}\lesssim W(x)$ and thus $\abs{W(x)-W(y)}\lesssim W(x)\abs{x-y}$ uniformly over $x,y\ge\underline\sigma^2>0$.

For the first assertion \eqref{EqAssert1} note that in ${\cal E}_2$ the estimator $\widehat{IV}_\eps$ is unbiased and 
\[ \Var\Big(\sum_{j\ge 1} w_{jk}(\sigma^2)h^{-2}j^2\pi^2(y_{jk}^2-\eps^2)\Big)=
\frac{2}{\sum_{j\ge 1}(\sigma^2(kh)+h_0^{-2}\pi^2j^2)^{-2}}
\]
such that by formula \eqref{Eqseries} and Riemann sum approximation as $h_0\to\infty$ (with arbitrary speed)
\[ \eps^{-1}\Var(\widehat{IV}_\eps)=\sum_{k=0}^{h^{-1}-1}\frac{2hh_0}{\sum_{j\ge 1}(\sigma^2(kh)+h_0^{-2}\pi^2j^2)^{-2}}\rightarrow 8\int_0^1\sigma^3(t)\,dt.
\]
Due to the independence and Gaussianity of the $(y_{jk})$ we deduce also
\[ \E\Big[\Big(\sum_{j\ge 1} w_{jk}(\sigma^2)h^{-2}j^2\pi^2(y_{jk}^2-\E[y_{jk}^2])\Big)^4\Big]\lesssim \Var\Big(\sum_{j\ge 1} w_{jk}(\sigma^2)h^{-2}j^2\pi^2(y_{jk}^2-\eps^2)\Big)^2
\]
such that the central limit theorem under a Lyapounov condition with power $p=4$ (e.g. \cit{Shiryaev}) proves assertion \eqref{EqAssert1}, assuming $h\to 0$ and $h_0\to\infty$. A feasible estimator is obtained by neglecting frequencies larger than some $J=J(\eps)$:
\begin{align}
\widehat{IV}_{\eps,J}&:=\sum_{k=0}^{h^{-1}-1} h\sum_{j=1}^J w_{jk}^J(\tilde\sigma_n^2)h^{-2}j^2\pi^2(y_{jk}^2-\eps^2)\label{EqIVespJ}\\
\text{where } w_{jk}^J(\sigma^2)&:=\frac{(\sigma^2(kh)+ h_0^{-2}\pi^2j^2)^{-2}} {\sum_{l=1}^J(\sigma^2(kh)+ h_0^{-2}\pi^2l^2)^{-2}}.\label{EqwjkJ}
\end{align}

A simple calculation yields $\E[\abs{\widehat{IV}_{\eps,J}-\widehat{IV}_{\eps}}^2]\lesssim \eps(h_0/J)^3$ such that for $h_0/J\to 0$ convergence in probability implies again by Slutsky's lemma
\[ \eps^{-1/2}(\widehat{IV}_{\eps,J}-IV)\xrightarrow{\cal L} N\Big(0,8\int_0^1\sigma^3(t)\,dt\Big).
\]

By the above argument, weak convergence results transfer from ${\cal E}_2$ to ${\cal E}_0$ and we obtain the following result where we give a concrete choice of the initial estimator, the block size $h$ and the spectral cut-off $J$ (we just need some consistent estimator $\tilde\sigma^2_n$, $h^{2\alpha} n^{1/2}\to 0$ as well as $hn^{1/2}\to\infty$ and $J^{-1}=o(h^{-1}n^{-1/2})$).

\begin{theorem}\label{ThmIVhat}
Let $y_{jk}^0$ for $j\ge 1$, $k=0,h^{-1}-1$ be the statistics \eqref{Eqyjk0} from model ${\cal E}_0$. For $h\thicksim n^{-1/2}\log(n)$ and $J/\log(n)\to\infty$  consider the estimator of integrated volatility
\[\widehat{IV}_{n}:=\sum_{k=0}^{h^{-1}-1} h\sum_{j=1}^J w_{jk}^J(\tilde\sigma_n^2)h^{-2}j^2\pi^2((y_{jk}^0)^2-\delta n^{-1})
\]
with weights $w_{jk}^J$ from \eqref{EqwjkJ} and the initial estimator $\tilde\sigma_n^2$ from \eqref{EqSigmaTilde}.
Then $\widehat{IV}_n$ is asymptotically efficient in the sense that
\[ n^{1/4}(\widehat{IV}_n-IV)\xrightarrow{\cal L} N\Big(0,8\delta\int_0^1\sigma^3(t)\,dt\Big)\text{ as }n\to\infty,
\]
provided $\sigma^2$ is strictly positive and $\alpha$-H\"older continuous with $\alpha>1/2$.
\end{theorem}

This might serve as a benchmark for more general models, whereas we, in the spirit of \cit{Mykland}, focus on elucidating the underlying fundamental structures.
In particular, we should dispense with the Gaussianity of the microstructure noise $(\eps_i)$ as well as with the deterministic nature of the volatility $\sigma^2$.
The analysis in both cases, however, cannot simply rely on model ${\cal E}_2$, since ${\cal E}_0$ is non-Gaussian. Different tools are required.

\section{Appendix}

\subsection{Gaussian measures, Hellinger distance and Hilbert-Schmidt norm}
\label{SecPrelim}

We gather basic facts about cylindrical Gaussian measures, the Hellinger distance and their interplay.

Formally, we realize the white noise experiments, as $L^2$-indexed Gaussian variables, e.g. in experiment ${\cal E}_1$ we observe for any $f\in L^2([0,1])$
\[ Y_f:=\scapro{f}{dY}:=\int_0^1 f(t) \Big(\int_0^t\sigma(s)dB(s)\Big)\,dt+\eps\int_0^1 f(t)\,dW_t.
\]
Canonically, we thus define $\PP^{\sigma,\eps}$ on the set $\Omega=\R^{L^2([0,1])}$ with product Borel $\sigma$-algebra ${\cal F}={\mathfrak B}^{\otimes L^2([0,1])}$ (realizing a cylindrical centered Gaussian measure). Its covariance structure is given by
\[ \E[Y_fY_g]=\scapro{Cf}{g},\quad f,g\in L^2([0,1]),\]
with the covariance operator $C:L^2([0,1])\to L^2([0,1])$ given by
\[
Cf(t)=\int_0^1 \Big(\int_0^{t\wedge u}\sigma^2(s)\,ds\Big)f(u)\,du+\eps^2f(t),\quad f\in L^2([0,1]).
\]
Note that $C$ is not trace class and thus does not define a Gaussian measure on $L^2([0,1])$ itself.

In the construction, it suffices to prescribe $(Y_{e_m})_{m\ge 1}$ for an orthonormal basis $(e_m)_{m\ge 1}$ and to set
\[ Y_f:=\sum_{m=1}^\infty \scapro{f}{e_m}Y_{e_m}.\]
This way, we can define $\PP^{\sigma,\eps}$ equivalently on the sequence space $\Omega=\R^{\N}$ with product $\sigma$-algebra ${\cal F}={\mathfrak B}^{\otimes \N}$. This is useful when extending results from finite dimensions.

The Hellinger distance between two probability measures $\PP$ and $\QQ$ on $(\Omega,{\cal F})$ is defined as
\[ H(\PP,\QQ)=\Big(\int_\Omega \big(\sqrt{p(\omega)}-\sqrt{q(\omega)}\big)^2\mu(d\omega)\Big)^{1/2},
\]
where $\mu$ denotes a dominating measure, e.g. $\mu=\PP+\QQ$, and $p$ and $q$ denote the respective densities. The total variation distance is smaller than the Hellinger distance:
\begin{equation}\label{EqHellTV}
\norm{\PP-\QQ}_{TV}\le H(\PP,\QQ).
\end{equation}
The identity $H^2(\PP,\QQ)=2-2\int\sqrt{p}\sqrt{q}d\mu$ implies the bound for finite or countably infinite product measures
\begin{equation}\label{EqHellProd} H^2\Big(\bigotimes_{n}\PP_n,\bigotimes_n\QQ_n\Big)\le \sum_n H^2(\PP_n,\QQ_n).
\end{equation}
Moreover, the Hellinger distance is invariant under bi-measurable bijections $T:\Omega\to\Omega'$ since with the densities $p\circ T^{-1}$, $q\circ T^{-1}$ of the image measures $\PP^T$ and $\QQ^T$ with respect to $\mu^T$ we have
\begin{equation}\label{EqHellTrafo}
H^2(\PP^T,\QQ^T)=\int_{\Omega'} (\sqrt{p\circ T^{-1}}-\sqrt{q\circ T^{-1}})^2d\mu^T=\int_\Omega (\sqrt{p}-\sqrt{q})^2 d\mu=H^2(\PP,\QQ).
\end{equation}
For the one-dimensional Gaussian laws $N(0,1)$ and $N(0,\sigma^2)$ we derive
\[ H^2(N(0,1),N(0,\sigma^2))=2-\sqrt{8\sigma/(\sigma^2+1)}\le 2(\sigma^2-1)^2.\]
For the multi-dimensional Gaussian laws $N(0,\Sigma_1)$ and $N(0,\Sigma_2)$ with invertible covariance matrices $\Sigma_1,\Sigma_2\in\R^{d\times d}$ we obtain by linear transformation and independence, denoting by $\lambda_1,\ldots,\lambda_d$ the eigenvalues of $\Sigma_1^{-1/2}\Sigma_2\Sigma_1^{-1/2}$:
\[ H^2(N(0,\Sigma_1),N(0,\Sigma_2))=H^2(N(0,\Id),N(0,\Sigma_1^{-1/2}\Sigma_2\Sigma_1^{-1/2}))
\le \sum_{k=1}^d 2(\lambda_k-1)^2.
\]
The last sum is nothing, but the squared Hilbert-Schmidt (or Frobenius norm) of $\Sigma_1^{-1/2}\Sigma_2\Sigma_1^{-1/2}-\Id$ such that
\begin{equation}\label{EqHellHS}
H^2(N(0,\Sigma_1),N(0,\Sigma_2))\le 2 \norm{\Sigma_1^{-1/2}(\Sigma_2-\Sigma_1)\Sigma_1^{-1/2}}_{HS}^2.
\end{equation}
Observing that \eqref{EqHellProd} and \eqref{EqHellTrafo} also apply to Gaussian measures on the sequence space $\R^{\N}$, the bound \eqref{EqHellHS} is also valid for (cylindrical) Gaussian measures $N(0,\Sigma_i)$ with self-adjoint positive definite covariance operators $\Sigma_i:L^2([0,1])\to L^2([0,1])$.

The Hilbert-Schmidt norm of a linear operator $A:H\to H$ on any separable real Hilbert space $H$ can be expressed by its action on an orthonormal basis $(e_m)$ via
\[ \norm{A}_{HS}^2=\sum_{m,n} \scapro{Ae_m}{e_n}^2,\]
which for a matrix is just the usual Frobenius norm. For self-adjoint operators $A,B$
with $\abs{\scapro{Av}{v}}\le \abs{\scapro{Bv}{v}}$ for all $v\in H$
we use the eigenbasis $(e_m)$ of $A$ and obtain
\begin{equation}\label{EqHSorder}
\norm{A}_{HS}^2=\sum_{m} \scapro{Ae_m}{e_m}^2\le \sum_{m,n} \scapro{Be_m}{e_n}^2
=\norm{B}_{HS}^2.
\end{equation}
Furthermore, it is straight-forward to see for any bounded operator $T$
\begin{equation}\label{EqHSTrafo}
\norm{TA}_{HS}\le \norm{T}\norm{A}_{HS},\quad \norm{AT}_{HS}\le \norm{T}\norm{A}_{HS}
\end{equation}
with the usual operator norm $\norm{T}$ of $T$. Finally, for integral operators $Kf(x)=\int_0^1 k(x,y)f(y)\,dy$ on $L^2([0,1])$ it is well known that
\begin{equation}\label{EqHSL2}
\norm{K}_{HS}=\norm{k}_{L^2([0,1]^2)}.
\end{equation}

For two Gaussian laws with different mean vectors $\mu_1,\mu_2$ and with the same invertible covariance matrix $\Sigma$ we can similarly use the transformation $\Sigma^{-1/2}$ and the scalar case $H^2(N(m_1,1),N(m_2,1))=2(1-e^{-(m_1-m_2)^2/8})\le (m_1-m_2)^2/4$ to conclude by independence
\begin{equation}\label{EqHellmean}
H^2(N(\mu_1,\Sigma),N(\mu_2,\Sigma))\le \tfrac14 \norm{\Sigma^{-1/2}(\mu_1-\mu_2)}^2.
\end{equation}
Combining \eqref{EqHellHS} and \eqref{EqHellmean} we obtain by the triangle inequality the bound
\begin{equation}\label{EqHellGaussGeneral}
H^2(N(\mu_1,\Sigma_1),N(\mu_2,\Sigma_2))\lesssim  \norm{\Sigma_1^{-1/2}(\mu_1-\mu_2)}^2
+\norm{\Sigma_1^{-1/2}(\Sigma_2-\Sigma_1)\Sigma_1^{-1/2}}_{HS}^2.
\end{equation}

\subsection{Proof of Theorem \ref{ThmRegression}}\label{SecProofThmRegression}

We first show that ${\cal E}_1$ is asymptotically at least as informative as ${\cal E}_0$ for $\eps=\delta/\sqrt{n}$ and $\alpha>0$. From ${\cal E}_1$ with $\eps=\delta/\sqrt{n}$ we can generate the observations (statistics)
\begin{align*}
\tilde Y_i&:=n\int_{(2i-1)/2n}^{(2i+1)/2n} dY_t=n\int_{(2i-1)/2n}^{(2i+1)/2n} X_tdt+\tilde\eps_i,\quad i=1,\ldots,n-1,\\
\tilde Y_n&:=2n\int_{(2n-1)/2n}^1 dY_t=2n\int_{(2n-1)/2n}^1 X_tdt+\tilde\eps_n,
\end{align*}
with $\tilde\eps_i=n\eps (W_{(2i+1)/2n}-W_{(2i-1)/2n})\sim N(0,\delta^2)$ and similarly $\tilde\eps_n\sim N(0,\delta^2)$, all independent. In contrast to standard equivalence proofs, it turns out to be essential here to take $\tilde Y_i$ as a mean symmetric around the point $i/n$. Since $(Y_i)$ and $(\tilde Y_i)$ are defined on the same sample space, using inequality \eqref{EqHellTV} it suffices to prove that the Hellinger distance between the law of $(Y_i)$ and the law of $(\tilde Y_i)$ tends to zero as $n$ tends to infinity.

For the integrated volatility function we introduce the notation
\[ a(t):=\int_0^t\sigma^2(s)\,ds,\quad 0\le t\le 1.\]
For notational convenience we also set $a(1+s):=a(1-s)$ for $s>0$.

The covariance matrix $\Sigma^Y$ of the centered Gaussian vector $(Y_i)$ is given by
\[ \Sigma^Y_{kl}:=\E[Y_kY_l]=a(k/n)+\delta^2{\bf 1}(k=l),\quad 1\le k\le l\le n.
\]
Similarly, the covariance matrix $\Sigma^{\tilde Y}$ of the centered Gaussian vector $(\tilde Y_i)$ is given by
\[ \Sigma^{\tilde Y}_{kl}:=\E[\tilde Y_k \tilde Y_l]=n\int_{(2k-1)/2n}^{(2k+1)/2n}a(t)\,dt+\delta^2{\bf 1}(k=l),\quad 1\le k\le l\le n,
\]
where for $k=l=n$ we used the convention for $a(1+s)$ above. We bound the  Hellinger distance using consecutively \eqref{EqHellHS}, $\Sigma^Y\ge \delta^2\Id$ in \eqref{EqHSorder} and \eqref{EqHellProd}, a Taylor expansion for $a$ and treating the case $k=l=n$ by a Lipschitz bound separately:
\begin{align*}
&H^2({\cal L}(Y_i, i=1,\ldots,n),{\cal L}(\tilde Y_i, i=1,\ldots,n))\\
&\le 2 \norm{(\Sigma^Y)^{-1/2}(\Sigma^Y-\Sigma^{\tilde Y})(\Sigma^Y)^{-1/2}}_{HS}^2\\
&\le 2\delta^{-4}\norm{\Sigma^{\tilde Y}-\Sigma^Y}_{HS}^2\\
&\le 4\delta^{-4}\sum_{1\le k\le l\le n} \Big(n\int_{(2k-1)/2n}^{(2k+1)/2n}(a(t)-a(k/n))\,dt\Big)^2\\
&\le 4\delta^{-4}\Big(O(R^2n^{-2})+n\sum_{k=1}^n \Big(n\int_{(2k-1)/2n}^{(2k+1)/2n}(a'(k/n)(t-k/n)+O(Rn^{-1-\alpha}))\,dt\Big)^2\Big)\\
&=4\delta^{-4}\Big(O(R^2n^{-2})+O(R^2n^{2-2-2\alpha})\Big)\\
&=O(\delta^{-4}R^2n^{-2\alpha}).
\end{align*}
Consequently, by \eqref{EqHellTV} the total-variation and thus also the Le Cam distance between the experiments of observing $(Y_i)$ and of observing $(\tilde Y_i)$ tends to zero for $n\to\infty$, which proves that the white noise experiment ${\cal E}_1$ is asymptotically at least as informative as the regression experiment ${\cal E}_0$.

To show the converse, we build from the regression experiment ${\cal E}_0$ a continuous time observation by linear interpolation. To this end we introduce the linear $B$-splines  (or hat functions) $b_i(t)=b(t-i/n)$ with $b(t)=\min(1+nt,1-tn){\bf 1}_{[-1/n,1/n]}(t)$ and set
\[ \hat Y'_t:=\sum_{i=1}^n Y_ib_i(t)=\sum_{i=1}^n X_{i/n}b_i(t)+\sum_{i=1}^n\eps_i b_i(t),\quad t\in [0,1].\]
Note that $(\hat Y'_t)$ is a centered Gaussian process with covariance function
\[ \hat c(t,s):=\E[\hat Y'_t\hat Y'_s]=\sum_{i,j=1}^n a((i\wedge j)/n)b_i(t)b_j(s)+\delta^2\sum_{i=1}^n b_i(t)b_i(s),\quad 0\le t,s\le 1.
\]
For any $f\in L^2([0,1])$ we thus obtain
\begin{align*}
\E[\scapro{f}{\hat Y'}^2]&=\sum_{i,j=1}^n a((i\wedge j)/n)\scapro{f}{b_i}\scapro{f}{b_j}+\delta^2\sum_{i=1}^n\scapro{f}{b_i}^2\\
&\le \sum_{i,j=1}^n a((i\wedge j)/n)\scapro{f}{b_i}\scapro{f}{b_j}+\delta^2n^{-1}\norm{f}^2,
\end{align*}
because $\int nb_i=1$ yields by Jensen's inequality $\scapro{f}{nb_i}^2\le \scapro{f^2}{nb_i}$ and we have $\sum_i b_i\le 1$. This means that the covariance operator $\hat C$ induced by the kernel $\hat c$ is smaller than
\[ \overline C f(t):=\sum_{i,j=1}^n a((i\wedge j)/n)\scapro{f}{b_j}b_i(t)+\delta^2n^{-1}f(t),\quad f\in L^2([0,1])
\]
in the sense that $\hat C-\overline C$ is positive (semi-)definite. Now observe that $\overline C$ is the covariance operator of the white noise observations
\begin{equation}\label{EqYbar}
d\bar Y_t= \sum_{i=1}^n X_{i/n}b_i(t)+\frac{\delta}{\sqrt{n}}dW_t,\quad t\in [0,1].
\end{equation}
Hence, we can generate these observations from $(\hat Y'_t)$ by randomisation, i.e. by adding uninformative $N(0,\overline C-\hat C)$-noise  to $\hat Y'$. Now it is easy to see that observing $\bar Y$ in \eqref{EqYbar} and $Y$ from ${\cal E}_1$ is asymptotically equivalent, since in terms of the respective covariance operators, using again \eqref{EqHellHS}, \eqref{EqHSorder} and \eqref{EqHellProd}, the squared Hellinger distance satisfies
\begin{align*}
H^2({\cal L}(\bar Y),{\cal L}(Y))&\le 2\norm{(C^Y)^{-1/2}(\overline C-C^Y)(C^Y)^{-1/2}}_{HS}^2\\
&\le 2\delta^{-4}n^2\int_0^1\int_0^1 \Big(a(t\wedge s)-\sum_{i,j=1}^n a((i\wedge j)/n)b_i(t)b_j(s)\Big)^2dtds\\
&= 2\delta^{-4}n^2\int_0^1\int_0^1 \Big(\sum_{i,j=0}^n (a(t\wedge s)-a((i\wedge j)/n))b_i(t)b_j(s)\Big)^2dtds,
\end{align*}
where for the last line we have used $\sum_{i=0}^nb_i(t)=1$ and $a(0)=0$. Since $b_i(t)\not=0$ can only hold when $i-\floor{nt}\in \{0,1\}$, the $\alpha$-H\"older regularity of $\sigma^2$ implies for $t\le s-1/n$:
\begin{align*}
&\Big(\sum_{i,j=0}^n (a(t\wedge s)-a((i\wedge j)/n))b_i(t)b_j(s)\Big)^2 \\
&=
\Big(\sum_{k,l=0}^1 (a'(\floor{nt}/n)(t-(k+\floor{nt})/n)+O(Rn^{-1-\alpha}))
b_{k+\floor{nt}}(t)b_{l+\floor{ns}}(s)\Big)^2\\
&=O(R^2n^{-2-2\alpha})+\Big(a'(\floor{nt}/n)\sum_{k=0}^1 (t-(k+\floor{nt})/n)
b_{k+\floor{nt}}(t)\Big)^2\\
&=O(R^2n^{-2-2\alpha}).
\end{align*}
A symmetric argument gives the same bound for $s\le t-1/n$. For $\abs{t-s}<1/n$ we use only the Lipschitz continuity of $a$ to obtain the bound $O(R^2n^{-2})$. Altogether we have found
\begin{align*}
H^2({\cal L}(\bar Y),{\cal L}(Y))& \le 2\delta^{-4}n^2\Big(O(R^2n^{-2-2\alpha})+n^{-1}O(R^2n^{-2})\Big)
=O(\delta^{-4}R^2n^{-2\alpha}),
\end{align*}
which together with the transformation in the other direction shows that the Le Cam distance between ${\cal E}_0$ and ${\cal E}_1$ is of order $O(\delta^{-2}Rn^{-\alpha})$.

\subsection{Proof of Proposition \ref{PropPiecewiseConstant}} \label{SecPropPiecewiseConstant}

The main tool is Proposition \ref{PropApprox} below. Together with the H\"older bound
\[\abs{\sigma^2(\floor{s}_h)-\sigma^2(s)}\le R h^\alpha,\quad s\in [0,1],
\]
it implies that for fixed $\sigma$ the observation laws in ${\cal E}_1$ and ${\cal E}_2$ have a Hellinger distance of order $Rh^\alpha\underline\sigma^{-3/2}\eps^{-1/2}$. By inequality \eqref{EqHellTV} this translates to the total variation and thus to the Le Cam distance.

\begin{proposition}\label{PropApprox}
For $\eps>0$ and continuous $\sigma:[0,1]\to(0,\infty)$ consider the law $\PP^{\sigma,\eps}$ generated by
\[ dY_t=\Big(\int_0^t\sigma(s)dB(s)\Big)\,dt+\eps\,dW_t,\quad t\in [0,1],\]
with independent Brownian motions $B$ and $W$. Then the Hellinger distance between two laws $\PP^{\sigma_1,\eps}$ and $\PP^{\sigma_2,\eps}$ satisfies
\[ H(\PP^{\sigma_1,\eps},\PP^{\sigma_2,\eps})\lesssim \norm{\sigma_1^2-\sigma_2^2}_\infty \big(\max_{t\in[0,1]} \sigma_1^{-3}(t)\big)\eps^{-1/2}.\]
\end{proposition}

\begin{proof}
The covariance operator $C_\sigma$ of $\PP^{\sigma,\eps}$ is for $f,g\in L^2([0,1])$ given by
\[ \scapro{C_\sigma f}{g} =\E[\scapro{f}{dY}\scapro{g}{dY}]=\E[\scapro{f}{X}\scapro{g}{X}]+\eps^2\scapro{f}{g}
=\int FG\sigma^2+\eps^2\int fg.
\]
For covariance operators corresponding to $\sigma_1$, $\sigma_2$ we have with $F(t)=-\int_t^1 f(s)ds$ by twofold partial integration
\begin{align*}
\abs{\scapro{(C_{\sigma_1}-C_{\sigma_2})f}{f}}&=\babs{\int_0^1\int_0^1 \int_0^{t\wedge s} (\sigma_1^2-\sigma_2^2)(u)\,du f(t)f(s)\,ds\,dt}\\
&=\babs{\int_0^1 F(u)^2(\sigma_1^2-\sigma_2^2)(u)\,du}\\
&\le \norm{\sigma_1^2-\sigma_2^2}_\infty\int_0^1 F(u)^2\,du
= \norm{\sigma_1^2-\sigma_2^2}_\infty \scapro{C_{BM}f}{f}
\end{align*}
with $C_{BM}g(t):=\int_0^1 (t\wedge s)g(s)\,ds$, the covariance operator of standard Brownian motion.
Using further the ordering $C_{\sigma_1}\ge \min_t\sigma_1^2(t)C_{BM}+\eps^2\Id$ and \eqref{EqHSorder}, \eqref{EqHellProd} we obtain
\begin{align*}
&\norm{C_{\sigma_1}^{-1/2}(C_{\sigma_2}-C_{\sigma_1})C_{\sigma_1}^{-1/2}}_{HS}\\
&\le \norm{\sigma_1^2-\sigma_2^2}_\infty \norm{C_{\sigma_1}^{-1/2}C_{BM}C_{\sigma_1}^{-1/2}}_{HS}\\
&\le \norm{\sigma_1^2-\sigma_2^2}_\infty
\norm{(\min_t\sigma_1^2(t) C_{BM}+\eps^2\Id)^{-1/2}C_{BM}(\min_t\sigma_1^2(t) C_{BM}+\eps^2\Id)^{-1/2}}_{HS}\\
&=\norm{\sigma_1^2-\sigma_2^2}_\infty
\norm{F(C_{BM})}_{HS},
\end{align*}
employing functional calculus with $F(x)=(\min_t\sigma_1^2(t)x+\eps^2)^{-1}x$.
The spectral properties of $C_{BM}$ imply that $F(C_{BM})$ has eigenfunctions $e_k(t)=\sqrt{2}\sin(\pi (k-1/2)t)$, $k\ge 1$, with eigenvalues $\lambda_k=\frac{4}{4\min_t\sigma_1^2(t)+(2k-1)^2\pi^2\eps^2}$, whence its Hilbert-Schmidt norm is of order $\max_t\sigma_1^{-3}(t)\eps^{-1/2}$. This yields the result.
\end{proof}

\subsection{Proof of Proposition \ref{PropE1G3loc}}\label{SecProofPropE1G3loc}

We only consider the case of odd indices $k$, both cases are treated analogously.
The result of Theorem 6.1 in conjunction with Theorem 5.2 of \cit{GramaNussbaum02} establishes that ${\cal E}_{3,m}^{odd}$ and the Gaussian regression experiment ${\cal G}_{3,m}$ of observing
\begin{equation}\label{EqG3m}
Y_k=v_\eps s^2(kh)+I(\sigma_0^2(kh))^{-1/2}\gamma_k,\quad k\in A_m\text{ odd},\quad\gamma_k\sim N(0,1)\text{ i.i.d.}
\end{equation}
are equivalent to experiments $\tilde{\cal E}_{3,m}=({\cal Y},{\cal G},(\tilde\PP^m_{s^2})_{s^2\in C_\alpha(R)})$ and $\tilde{\cal G}_{3,m}=({\cal Y},{\cal G},(\tilde\QQ^m_{s^2})_{s^2\in C_\alpha(R)})$, respectively, on the same space $({\cal Y},{\cal G})$ such that
\begin{equation}\label{EqHellLoc}
\sup_{s^2\in C_\alpha(R)}H^2(\tilde\PP_{s^2}^m,\tilde\QQ_{s^2}^m)\lesssim \ell^{-2\rho}
\end{equation}
holds for all $\rho<1$.

To be precise, it must be checked that the regularity conditions $R1-R3$ of \cit{GramaNussbaum02} are satisfied for all values $\delta$. One complication is that in our parametric model the probabilities $\PP_\theta$ and the Fisher information $I(\theta)$ depend on $h_0$ which tends to infinity. Yet, inspecting the proofs it becomes clear that the results remain valid if (a) the conditions $R1-R3$ hold for varying models, but with uniform constants and (b) the Fisher information is renormalized by the localisation such that the parametric rate $\ell^{-1/2}$ (in our block length notation) is attained. From the fact that $\PP_\theta$ is the product of one-dimensional exponential family models we easily check condition $R1$ for $\delta=1$ and condition $R2$ for any $\delta>0$. Both conditions hold uniformly over $h_0$ once the score $\dot l$ has been renormalized through multiplication by $h_0^{-1/2}$. In \eqref{EqItheta} we have already calculated the Fisher information and we infer directly condition $R3$ that $h_0^{-1}I(\theta)$ is uniformly bounded away from zero and infinity. We thus infer \eqref{EqHellLoc}.

In view of the independence among the experiments $({\cal E}_{3,m}^{odd})_m$ and equally among the experiments $({\cal G}_{3,m})_m$ we infer from \eqref{EqHellLoc} and \eqref{EqHellProd}
\[ \sup_{s^2\in C_\alpha(R)}H^2(\otimes_{m=1}^{(\ell h)^{-1}}\tilde\PP_{s^2}^m, \otimes_{m=1}^{(\ell h)^{-1}}\tilde\QQ_{s^2}^m)\lesssim (\ell h)^{-1} \ell^{-2\rho} \lesssim  \eps^{-1} v_\eps^2 h_0^{2\rho}v_\eps^{4\rho}.
\]
Since we assume $h_0=o(\eps^{(1-2\alpha)/2\alpha})$, the right-hand side tends to zero provided
\[ -1+ 2\frac{\alpha}{2\alpha+1}+\frac{\rho(1-2\alpha)}{\alpha}+\frac{4\rho\alpha}{2\alpha+1}
=\frac{\rho-\alpha}{\alpha(2\alpha+1)}>0
\]
holds. Since $\rho<1$ is arbitrary, this is always satisfied for $\alpha<1$. In the case $\alpha=1$ we use $h_0\lesssim \eps^{-p}$ for some $p<1/2$. We have derived asymptotic equivalence between
the product experiments $\otimes_m\tilde{\cal E}_{3,m}^{loc}$ and $\otimes_m\tilde{\cal G}_{3,m}$. A fortiori, applying the \cit{BrownLow} result, this leads to asymptotic equivalence between
observing $(y_{jk})$ in experiments ${\cal E}_{2,loc}$ and the corresponding Gaussian shift models of observing
\begin{equation}\label{EqG2.9loc}
dY_t=I(\sigma_0^2(t))^{1/2}v_\eps s^2(t)\,dt+ (2h)^{1/2}dW_t,\quad t\in [0,1].
\end{equation}

From the explicit form \eqref{EqItheta} of the Fisher information we infer  for $h_0\to\infty$
\[\babs{\frac{2\theta^{3/2}}{h_0}I(\theta)-\frac14+\frac{1}{2\theta^{1/2} h_0}}\lesssim e^{-\underline\sigma h_0}.\]
Consequently, by the polynomial growth of $h_0$ in $\eps^{-1}$, the Kullback-Leibler divergence between the observation laws from \eqref{EqG2.9loc} and the model ${\cal G}_{3,loc}$ converges to zero. This gives the result.

\subsection{Proof of Proposition \ref{PropE1locyjk}}\label{ProofPropE1locyjk}

Since the observations $y_{jk}$ for $j\ge 1$ are the same in $\cal Y$ and $\tilde{\cal Y}$, we can work conditionally on those. Moreover, it suffices to consider only the event $\Omega_\eps:=\{\norm{\hat\sigma_\eps^2-\sigma^2}_\infty\le Rv_\eps\}$ because the squared Hellinger distance satisfies (with obvious notation)
\begin{align*}
&H^2({\cal L}({\cal Y}),{\cal L}(\tilde{\cal Y}))=\E[H^2({\cal L}((y_{0k})_k\,|\,(y_{jk})_{j\ge 1,k}),{\cal L}((\tilde y_{0k})_k\,|\,(y_{jk})_{j\ge 1,k}))]\\
&\le \E[H^2({\cal L}((y_{0k})_k\,|\,(y_{jk})_{j\ge 1,k}),{\cal L}((\tilde y_{0k})_k\,|\,(y_{jk})_{j\ge 1,k})){\bf 1}_{\Omega_\eps}]+2\PP(\Omega_\eps^\complement)
\end{align*}
with $\PP(\Omega_\eps^\complement)\to 0$. Conditional on $(y_{jk})_{j\ge 1,k}$, both laws are Gaussian,  $(y_{0,k})_k$ has
mean $\mu$ with
\[ \mu_0=2\sum_{j\ge 1}\frac{\Var(\beta_{jk})}{\Var(y_{jk})}y_{jk},\quad \mu_k=\sum_{j\ge 1}\frac{\Var(\beta_{jk})}{\Var(y_{jk})}\big((-1)^{j+1}y_{j,k-1}+y_{jk}\big),\quad k\ge 1,
\]
and covariance matrix $\Sigma$ with
\[ \Sigma_{k,k'}=\begin{cases}
2c_k\sum_{j\ge 1}\frac{\eps^2\Var(\beta_{jk})}{\Var(y_{jk})}+\eps^2,& \text{if } k'=k,\\
c_{k\wedge k'}\sum_{j\ge 1}\frac{\eps^2\Var(\beta_{jk})}{\Var(y_{jk})}
-\frac{\eps^2}{2},& \text{if } \abs{k'-k}=1,\\
0,&\text{otherwise,}
\end{cases}
\]
where $c_k:=1\vee(2-k)\in\{1,2\}$. Conditional mean $\tilde\mu$ and covariance matrix $\tilde\Sigma$ of $(\tilde y_{0k})_k$ have the same representation, but replacing $\Var$ each time by $\Var_\eps$.

From the tri-diagonal structure of $\Sigma$ and from
\[\sum_{j\ge 1}\frac{\Var(\beta_{jk})}{\Var(y_{jk})}\thicksim \sum_{j\ge 1}\frac{(h_0/j)^2}{(h_0/j)^2+1}\thicksim h_0,\quad h_0\to\infty,
\]
we infer $\Sigma\gtrsim (\eps^2h_0+\eps^2)\Id\ge \eps h\Id$ in matrix order.
Combining this with the Hellinger bound \eqref{EqHellGaussGeneral} we arrive at the estimate
\begin{align*}
&\E[H^2({\cal L}((y_{0k})_k\,|\,(y_{jk})_{j\ge 1,k}),{\cal L}((\tilde y_{0k})_k\,|\,(y_{jk})_{j\ge 1,k}))]\\
&\lesssim \E\Big[\frac{\norm{\mu-\tilde\mu}^2}{\eps h}\Big]+\frac{\norm{\Sigma-\tilde\Sigma}_{HS}^2} {\eps^2h^2}\\
&\lesssim
\sum_{j\ge 1,k} \Big(\frac{\Var(\beta_{jk})}{\Var(y_{jk})}
-\frac{\Var_\eps(\beta_{jk})}{\Var_\eps(y_{jk})}\Big)^2\frac{\Var(y_{jk})}{\eps h}+
\sum_{j\ge 1,k} \Big(\frac{\eps^2\Var(\beta_{jk})}{\Var(y_{jk})}
-\frac{\eps^2\Var_\eps(\beta_{jk})}{\Var_\eps(y_{jk})}\Big)^2\eps^{-2}h^{-2}.
\end{align*}
The function $G(z):=\frac{\norm{\Phi_{jk}}^2z}{\norm{\Phi_{jk}}^2z+\eps^2}$ has derivative $G'(z)=\frac{\norm{\Phi_{jk}}^2\eps^2}{(\norm{\Phi_{jk}}^2z+\eps^2)^2}$ and thus satisfies uniformly over all $z$ bounded away from zero
$\abs{G(w)-G(z)}\lesssim \frac{\norm{\Phi_{jk}}^2\eps^2\abs{w-z}}{(\norm{\Phi_{jk}}^2+\eps^2)^2}$. Inserting $\abs{\sigma^2-\sigma_0^2}\lesssim v_\eps$ and $\norm{\Phi_{jk}}\thicksim h/j$, we thus find the uniform bound on $\Omega_\eps$
\[ \Big(\frac{\Var(\beta_{jk})}{\Var(y_{jk})}
-\frac{\Var_\eps(\beta_{jk})}{\Var_\eps(y_{jk})}\Big)^2\lesssim \frac{v_\eps^2\eps^4h^4/j^4}{(\eps^2+h^2/j^2)^4}\thicksim v_\eps^2\min(h_0/j,j/h_0)^4.
\]
Putting the estimates together, we arrive at
\begin{align*}
H^2({\cal L}({\cal Y}),{\cal L}(\tilde{\cal Y}))&\lesssim
v_\eps^2\sum_{j\ge 1,k} \min(h_0/j,j/h_0)^4\Big(\frac{1+h_0^2/j^2}{h_0}+\frac{1}
{h_0^2}\Big)+\PP(\Omega_\eps^\complement)\\
&\le 2v_\eps^2h^{-1}\sum_{j\ge 1}\min(h_0/j,j/h_0)^2h_0^{-1}+\PP(\Omega_\eps^\complement)\\
&\thicksim  v_\eps^2h_0^{-1}\eps^{-1}+\PP(\Omega_\eps^\complement)
\end{align*}
such that the Hellinger distance tends to zero uniformly if $h_0^{-1}v_\eps^2=o(\eps)$, which is ensured by our choice of $h_0$. This implies asymptotic equivalence of observing $\cal Y$ and $\tilde{\cal Y}$ and thus of experiment ${\cal E}_2$ and of just observing $(y_{jk})_{j\ge 1,k}$ in ${\cal E}_2$. By independence the latter is equivalent to ${\cal E}_{2,odd}\otimes{\cal E}_{2,even}$.

\subsection{An explicit series representation}\label{SecSeries}

We aim at deriving the formula
\begin{equation}\label{Eqseries} \sum_{j\ge 1}\frac{\lambda^3}{(\lambda^2+\pi^2j^2)^2}=\frac{1+4\lambda e^{-2\lambda}-e^{-4\lambda}}{4(1-e^{-2\lambda})^2}
-\frac{1}{2\lambda}
\end{equation}
for any $\lambda>0$. We employ Fourier techniques and consider the Fourier coefficients of $g(x)=e^{-\lambda x/\pi}$:
\[\hat{g}(j):=(2\pi)^{-1/2}\int_0^{2\pi}g(x)e^{ijx}dx
=\frac{\pi(1-e^{-2\lambda})}{\sqrt{2\pi}(\lambda-i\pi j)},\quad j\in\Z.
\]
For the $2\pi$-periodic convolution $g\ast g(x)=xe^{-\lambda x/\pi}+(2\pi-x)e^{-\lambda(2+x/\pi)}$
we obtain the Fourier coefficient as a product:
\[ \widehat{g\ast g}(j)=\frac{\pi^2(1-e^{-2\lambda})^2}{\sqrt{2\pi}(\lambda-i\pi j)^2}.
\]
The Parseval formula therefore yields
\[ \sum_{j\in\Z}\frac{1}{(\lambda^2+\pi^2j^2)^2}=\frac{2}{\pi^3(1-e^{-2\lambda})^4}
\sum_{j\in\Z}\abs{\widehat{g\ast g}(j)}^2=\frac{2}{\pi^3(1-e^{-2\lambda})^4} \norm{g\ast g}_{L^2}^2.
\]
We infer that $(\frac{\lambda}{\pi})^3\norm{g\ast g}_{L^2}^2$ equals
\begin{align*}
&\Big(\frac{\lambda}{\pi}\Big)^3 \int_0^{2\pi}\Big((1-e^{-2\lambda})^2x^2+4\pi^2e^{-4\lambda}
+4\pi(e^{-2\lambda}-e^{-4\lambda})x\Big)e^{-2\lambda x/\pi}dx\\
&=\Big(\lambda e^{-2\lambda}+(1-e^{-4\lambda})/4\Big)(1-e^{-2\lambda})^2
\end{align*}
and thus obtain
\[ \sum_{j\in\Z}\frac{\lambda^3}{(\lambda^2+\pi^2j^2)^2}=\frac{2}{(1-e^{-2\lambda})^2}
\Big(\lambda e^{-2\lambda}+(1-e^{-4\lambda})/4\Big).
\]
Using the symmetry in $j$, we establish \eqref{Eqseries}.

\subsection*{Acknowledgement}
I am grateful to Marc Hoffmann, Mark Podolskij and Johannes Schmidt-Hieber for very
useful discussions.

\bibliographystyle{economet}

\bibliography{mybib2}

\end{document}